\documentclass[12pt]{article}
\usepackage{geometry}
\usepackage{amsfonts}
\usepackage[latin1]{inputenc}
\usepackage{amssymb , amsmath, amsthm, amsopn, amstext, amscd}
\usepackage[T1]{fontenc}
\usepackage{latexsym, enumerate}
\usepackage{graphics}
\usepackage{graphicx}

\newtheorem{thm}{Theorem}
\newtheorem{lm}{Lemma}
\newtheorem{rmq}{Remark}

\DeclareMathOperator*{\argmin}{Argmin}
\DeclareMathOperator*{\pen}{pen}
\DeclareMathOperator*{\diag}{diag}

\newcommand{\abs}[1]{\left\vert#1\right\vert}

\title{Regularization with the Smooth-Lasso procedure}
\author{\noindent{ Mohamed Hebiri\footnote{hebiri@math.jussieu.fr}} \\ \\
\small{Laboratoire de Probabilit\'{e}s et Mod\`{e}les Al\'{e}atoires, CNRS-UMR 7599,}\\ 
\small{Universit\'{e} Paris 7 - Diderot, UFR de Math\'{e}matiques,}\\ 
\small{175 rue de Chevaleret F-75013 Paris, France.}}

\date{}

\begin{document}

\maketitle

\begin{abstract}
We consider the linear regression problem. We propose the S-Lasso procedure to estimate the unknown regression parameters. This estimator enjoys sparsity of the representation while taking into account correlation between successive covariates (or predictors). The study covers the case when $p\gg n$, i.e. the number of covariates is much larger than the number of observations. In the theoretical point of view, for fixed $p$, we establish asymptotic normality and consistency in variable selection results for our procedure. When $p\geq n$, we provide variable selection consistency results and show that the S-Lasso achieved a Sparsity Inequality, i.e., a bound in term of the number of non-zero components of the oracle vector. It appears that the S-Lasso has nice variable selection properties compared to its challengers. Furthermore, we provide an estimator of the effective degree of freedom of the S-Lasso estimator. A simulation study shows that the S-Lasso performs better than the Lasso as far as variable selection is concerned especially when high correlations between successive covariates exist. This procedure also appears to be a good challenger to the Elastic-Net~\cite{Zou-E-Net}.
\end{abstract}

\section{Introduction}

We focus on the usual linear regression model:
\begin{equation} \label{eq_depart}
y_{i}= x_{i} \beta^{*} + \varepsilon_{i}, \quad \quad i=1,\ldots,n,
\end{equation}
where the design $x_{i}=(x_{i,1},\ldots,x_{i,p}) \in \mathbb{R}^p$ is deterministic, $\beta^*=(\beta^*_1,\ldots,\beta^*_p)' \in \mathbb{R}^p$ is the unknown parameter and $\varepsilon_1,\ldots,\varepsilon_n,$ are independent identically distributed (i.i.d.) centered Gaussian random variables with known variance $\sigma^{2}$. We wish to estimate $\beta^{*}$ in the sparse case, that is when many of its unknown components equal zero. Thus only a subset of the design covariates $(\xi_{j})_{j}$ is truly of interest where $\xi_{j} = (x_{1,j},\ldots,x_{n,j})',\,j=1,\ldots,p$. Moreover the case $p\gg n$ is not excluded so that we can consider $p$ depending on $n$. In such a framework, two main issues arise: i) the interpretability of the resulting prediction; ii) the control of the variance in the estimation. Regularization is therefore needed. For this purpose we use selection type procedures of the following form:
\begin{equation}\label{eq:penalized-risk}
\tilde \beta=\underset{\beta \in \mathbb{R}^{p}}{\argmin}\left\lbrace \Vert Y- X\beta\Vert_n^2 + \pen(\beta) \right \rbrace,
\end{equation}
where $X=(x_{1}',\ldots,x_{n}')'$, $Y=(y_{1},\ldots,y_{n})'$ and $\pen:  \, \mathbb{R}^p \rightarrow \mathbb{R}$ is a positive convex function called the penalty. For any vector $a=(a_{1},\ldots,a_{n})'$, we have adopted the notation $\Vert a \Vert_n^2 = n^{-1} \sum_{i=1}^n |a_{i}|^{2}$ (we denote by $<\cdot,\cdot>_n$ the corresponding inner product in $\mathbb{R}^{n}$). The choice of the penalty appears to be crucial. Although well-suited for variable selection purpose, Concave-type penalties (\cite{Fan&Li-SCAD}, \cite{Tsyb-VanDeGeer-SquareRoot} and \cite{ArnakTsyb}) are often computationally hard to optimize. Lasso-type procedures (modifications of the $l_{1}$ penalized least square (Lasso) estimator introduced by Tibshirani \cite{Tibshirani-LASSO}) have been extensively studied during the last few years. Between many others, see \cite{Bunea_consist, SparLassBTW07, Zhao-Yu-Consist-Lasso} and references inside. Such procedures seem to respond to our objective as
they perform both regression parameters estimation and variable selection with low computational cost. We will explore this type of procedures in our study.

In the paper, we propose a novel modification of the Lasso we call the \textit{Smooth-lasso} (\textit{S-lasso}) estimator. It is defined as the solution of the optimization problem~\eqref{eq:penalized-risk} when the penalty function is a combination of the Lasso penalty (i.e., $\sum_{j=1}^{p} |\beta_j|$) and the $l_2$-fusion penalty (i.e., $\sum_{j=2}^{p}\left(\beta_{j}-\beta_{j-1}\right)^{2}$). The $l_2$-fusion penalty was first introduced in \cite{Fusion}. We add it to the Lasso procedure in order to overcome the variable selection problems observed by the Lasso estimator. Indeed the Lasso estimator has good selection properties but fails in some situations. More precisely, in several works (\cite{Bunea_consist, KarimNormSup, Meins-Buhl-Lasso-Graph, WainSelection, Yuan-Lin-Garrotte, Zhao-Yu-Consist-Lasso, Zou-Adaptive-Lasso} among others) conditions for the consistency in variable selection of the Lasso procedure are given. It was shown that the Lasso is not consistent when high correlations exist between the covariates. We give similar consistency conditions for the S-Lasso procedure and show that it is consistent in variable selection in much more situations than the Lasso estimator. From a practical point of view, problems are also encountered when we solve the Lasso criterion with the Lasso modification of the LARS algorithm \cite{Efron-LARS}. Indeed this algorithm tends to select only one representing covariates in each group of correlated covariates. We attempt to respond to this problem in the case where the covariates are ranked so that high correlations can exist between successive covariates. We will see through simulations that such situations support the use of the \textit{S-lasso} estimator. This estimator is inspired by the \textit{Fused-Lasso} \cite{Rosset-Fused}. Both S-Lasso and Fused-Lasso combine a $l_{1}$-penalty with a fusion term \cite{Fusion}. The fusion term is suggested to catch correlations between covariates. More relevant covariates can then be selected due to correlations between them. The main difference between the two procedures is that we use the $l_{2}$ distance between the successive coefficients (i.e., the $l_2$-fusion penalty) whereas the Fused-Lasso uses the $l_{1}$ distance (i.e., the $l_1$-fusion penalty: $\sum_{j=2}^{p}|\beta_{j}-\beta_{j-1}|$). Hence, compared to the Fused-Lasso, we sacrifice sparsity between successive coefficients in the estimation of $\beta^*$ in favor of an easier optimization due to the strict convexity of the $l_{2}$ distance. However, since sparsity is yet ensured by the Lasso penalty. The $l_2$-fusion penalty helps us to catch correlations between covariates. Consequently, even if there is no perfect match between successive coefficients our result are still interpretable. Moreover, when successive coefficients are significantly different, a perfect match seems to be not really adapted. In the theoretical point of view, The $l_2$ distance also helps us to provide theoretical properties for the S-Lasso which in some situations appears to outperforms the Lasso and the Elastic-Net \cite{Zou-E-Net}, another Lasso-type procedure. Let us mention that variable selection consistency of the Fused-Lasso and the corresponding Fused adaptive Lasso has also been studied in \cite{Rinaldo08FusedAdaptive} but in a different context from the one in the present paper. The result obtained in~\cite{Rinaldo08FusedAdaptive} are established not only under the sparsity assumption, but the model is also supposed to be {\it blocky}, that is the non-zero coefficients are represented in a block fashion with equal values inside each block.

Many techniques have been proposed to solve the weaknesses of the Lasso. The Fused-Lasso procedure is one of them and we give here some of the most popular methods; the Adaptive Lasso was introduced in~\cite{Zou-Adaptive-Lasso}, which is similar to the Lasso but with adaptive weights used to penalize each regression coefficient separately. This procedure reaches 'Oracles Properties' (i.e. consistency in variable selection and asymptotic normality). Another approach is used in the Relaxed Lasso \cite{Relax-Lasso} and aims to doubly-control the Lasso estimate: one parameter to control variable selection and the other to control shrinkage of the selected coefficients. To overcome the problem due to the correlation between covariates, group variable selection has been proposed by Yuan and Lin \cite{Yuan-Lin-GroupSelection} with the Group-Lasso procedure which selects groups of correlated covariates instead of single covariates at each step. A first step to the consistency study has been proposed in \cite{BachGpLasso} and Sparsity Inequalities were given in \cite{NousSparsGpLasso}. Another choice of penalty has been proposed with the Elastic-Net \cite{Zou-E-Net}. It is in the same spirit that we shall treat the S-Lasso from a some theoretical point of view.

The paper is organized as follows. In the next section, we present one way to solve the S-Lasso problem with the attractive property of piecewise linearity of its regularization path. Section~\ref{theor} gives theoretical performances of the considered estimator such as consistency in variable selection and asymptotic normality when $p\leq n$ whereas consistency in estimation and variable selection in the high dimensional case are considered in Section~\ref{sec:theoryPgrd}. We also give an estimate of the effective degree of freedom of the S-Lasso estimator in Section~\ref{sec-Selection}. Then, we provide a way to control the variance of the estimator by scaling in Section~\ref{sec-Norm} where a connection with soft-thresholding is also established. A generalization and comparative study to the Elastic-Net is done in Section~\ref{sec:Ext}. We finally give experimental results in Section~\ref{sec-exper} showing the S-Lasso performances against some popular methods. All proofs are postponed to an Appendix section.

\section{The S-Lasso procedure}\label{solve}
As described above, we define the S-Lasso estimator $\hat{\beta}^{SL}$ as the solution of the optimization problem~\eqref{eq:penalized-risk} when the penalty function is:
\begin{equation}\label{critere_S-lasso}
\pen(\beta) = \lambda | \beta |_{1} +
\mu\sum_{j=2}^{p}\left(\beta_{j}-\beta_{j-1}\right)^{2},
\end{equation}
where $\lambda$ and $\mu$ are two positive parameters that control the smoothness of our estimator. For any vector $a=(a_{1},\ldots,a_{p})'$, we have used the notation $| a |_{1} = \sum_{j=1}^{p}\abs{a_{j}}$. Note that when $\mu = 0$, the solution is the Lasso estimator so that it appears as a special case of the S-Lasso estimator. Now we deal with the resolution of the S-Lasso problem \eqref{eq:penalized-risk}-\eqref{critere_S-lasso} and its computational cost. From now on, we suppose w.l.o.g. that $X=(x_{1},\ldots,x_{n})'$ is standardized (that is $n^{-1}\sum_{i=1}^{n}x_{i,j}^{2}= 1$ and $n^{-1}\sum_{i=1}^{n}x_{i,j} = 0 $) and $Y=(y_{1},\ldots,y_{n})'$ is centered (that is $n^{-1}\sum_{i=1}^{n}y_{i} = 0 $). The following lemma shows that the S-Lasso criterion can be expressed as a Lasso criterion by augmenting the data artificially.
\begin{lm}\label{equivalenceLasso}
Given the data set $(X,Y)$ and $(\lambda,\mu)$. Define the extended dataset
$(\widetilde{X},\widetilde{Y})$ by
$$\widetilde{X}=\frac{1}{\sqrt{1+\mu }} \begin{pmatrix}
X \\
\sqrt{n \mu}\mathbf{J}
\end{pmatrix} \quad
\text{and} \quad \widetilde{Y}=
\begin{pmatrix}
Y \\
\mathbf{0}
\end{pmatrix},$$ where $\mathbf{0}$ is a vector of size $p$ containing only zeros and $\mathbf{J}$ is the $p\times p$ matrix
\begin{equation}\label{Jmatrix}
\mathbf{J}=\begin{pmatrix}
  0 & 0 & 0 & \ldots & 0 \\
  1 & -1 & \ddots & \ddots & \vdots \\
  0 & 1 & -1 & \ddots & 0 \\
  \vdots & \ddots &\ddots & \ddots & 0 \\
  0 & \ldots & 0 & 1 & -1
\end{pmatrix}.
\end{equation}
Let $r=\lambda / \sqrt{1+\mu }$ and $b= \sqrt{1+\mu } \, \beta$. Then the
S-Lasso criterion can be written
\begin{equation*}
\left\|\widetilde{Y}-\widetilde{X} b\right\|_{n}^{2}+ r | b |_{1}.
\end{equation*}
Let $\hat{b}$ be the minimizer of this Lasso-criterion, then
\begin{equation*}
\hat{\beta}^{SL}=\frac{1}{\sqrt{1+\mu}}\,\hat{b}.
\end{equation*}
\end{lm}
This result is a consequence of simple algebra. Lemma~\ref{equivalenceLasso}
motivates the following comments on the S-Lasso procedure.

\begin{rmq}[{\it Regularization paths}]\label{rq:RegulPath}  The S-Lasso modification of the LARS is an
iterative algorithm. For a fixed $\mu$ (appearing \eqref{critere_S-lasso}), it
constructs at each step an estimator based on the correlation between covariates
and the current residue. Each step corresponds to a value of $\lambda$. Then
for a fixed $\mu$, we get the evolution of the S-Lasso estimator coefficients
values when $\lambda$ varies. This evolution describes the regularization paths
of the S-Lasso estimator which are piecewise linear
\cite{Rosset-PeacewiseLin}. This property implies that the S-Lasso problem can
be solved with the same computational cost as the ordinary least square (OLS)
estimate using the Lasso modification version of the LARS algorithm.
\end{rmq}

\begin{rmq}[{\it Implementation}] The number of covariates that the LARS algorithm and
its Lasso version can select is limited by the number of rows in the matrix
$X$. Applied to the augmented data $(\widetilde{X},\widetilde{Y})$ introduced
in Lemma~\ref{equivalenceLasso}, the Lasso modification of the LARS algorithm
is able to select all the $p$ covariates. Then we are no longer limited by the
sample size as for the Lasso \cite{Efron-LARS}.
\end{rmq}

\section{Theoretical properties of the S-Lasso estimator when $p\leq n$}\label{theor}
In this section we introduce the theoretical results according to the S-Lasso with a moderate sample size ($p \leq n$). We first provide rates of convergence of the S-Lasso estimator and show how through a control on the regularization parameters we can establish root-$n$ consistency and asymptotic normality. Then we look for variable selection consistency. More precisely, we give conditions under which the S-Lasso estimator succeeds in finding the set of the non-zero regression coefficients. We show that with a suitable choice of the tuning parameter $(\lambda,\mu)$, the S-Lasso is consistent in variable selection. All the results of this section are proved in Appendix~A.

\subsection{Asymptotic Normality}
In this section, we allow the tuning parameters $(\lambda,\mu)$ to depend on the sample size $n$. We emphasize this dependence by adding a subscript $n$ to these parameters. We also fix the number of covariates $p$. Let us note $\mathbb{I}(\cdot)$ the indicator function and define the sign function such that for any $x\in\mathbb{R}$, $\mathop{\rm Sgn}(x)$ equals $1,\,-1$ or $0$ respectively when $x$ is bigger, smaller or equals $0$. Knight and Fu \cite{Knight&Fu-Lasso-distrib} gave the asymptotic distribution of the Lasso estimator. We provide here the asymptotic distribution to the S-Lasso. Let $\mathbf{C}_n=n^{-1}X'X$, be Gram matrix, then
\begin{thm}\label{thm:AsymptoticNormality}
Given the data set $(X,Y)$, assume the correlation matrix verifies
\begin{equation*}
\mathbf{C}_n  \rightarrow \mathbf{C},\quad \text{when}\,\, n\rightarrow \infty,
\end{equation*}
in probability where $\mathbf{C}$ is a positive definite matrix. If there exists a sequence $v_n$ such that $v_n \rightarrow 0$ and the regularization parameters verify $\lambda_{n}v_{n}^{-1}\rightarrow \lambda \ge 0$ and $\mu_{n}v_{n}^{-1}\rightarrow \mu \ge 0$. Then, if $(\sqrt{n} v_{n})^{-1}\rightarrow \kappa \ge 0$, we have
\begin{equation*}
v_{n}^{-1}(\hat{\beta}^{SL}-\beta^{*})\xrightarrow{\mathcal{D}}\argmin_{u\in\mathbb{R}^{p}}V(u),\quad
\text{when}\,\, n\rightarrow \infty,
\end{equation*}
where
\begin{eqnarray*}
V(u)=-2\kappa u^{T}W+u^{T}\mathbf{C}u & + &
\lambda\sum_{j=1}^{p}\left\{u_{j}\mathop{\rm
Sgn}(\beta_{j}^{*})\mathbb{I}(\beta_{j}^{*}\neq 0) + \abs{u_{j}}\mathbb{I}(\beta_{j}^{*}=0)\right\} \\
& + &
2\mu\sum_{j=2}^{p}\left\{(u_{j}-u_{j-1})(\beta_{j}^{*}-\beta_{j-1}^{*})\mathbb{I}(\beta_{j}^{*}\neq\beta_{j-1}^{*})\right\},
\end{eqnarray*}
with $W\sim\mathcal{N}(0,\sigma^{2}\mathbf{C})$.
\end{thm}
\begin{rmq}
{\it When $\kappa \neq 0$ is a finite constant:} in this case $v_{n}^{-1}$ is $\mathcal{O}(\sqrt{n})$ so that the estimator $\hat{\beta}^{SL}$ is root-$n$ consistent. Moreover when $\lambda=\mu=0$, we obtain the following standard regressor asymptotic normality: $\displaystyle{\sqrt{n}(\hat{\beta}^{SL}-\beta^*)\xrightarrow{\mathcal{D}}\mathcal{N}(0,\sigma^{2}\mathbf{C}^{-1})}$.\\
{\it When $\kappa = 0$:} in this case, the rate of convergence is slower than $\sqrt{n}$ so that we no longer have the optimal rate. Moreover the limit is not random anymore.
\end{rmq}
Note first that the correlation penalty does not alter the asymptotic bias when successive regression coefficients are equal. We also remark that the sequence $v_n$ must be chosen properly as it determines our convergence rate. We would like $v_n$ to be as close as possible to $1/\sqrt{n}$. This sequence is calibrated by the user such that $\lambda_{n}/ v_n \rightarrow \lambda$ and $\mu_{n}/ v_n \rightarrow \mu$.

\subsection{Consistency in variable selection}\label{sec:Consistency}
In this section, variable selection consistency of the S-Lasso estimator is considered. For this purpose, we introduce the following sparsity sets: $\mathcal{A}^{*} = \{j:\,\beta_{j}^{*} \neq 0 \}$ and $\mathcal{A}_{n} = \{j:\,\hat{\beta}_{j}^{SL} \neq 0\}$. The set $\mathcal{A}^{*}$ consists of the non-zero coefficients in the vector of the oracle regression vector $\beta^{*}$. The set $\mathcal{A}_{n}$ consists of the non-zero coefficients in the S-Lasso estimator $\hat{\beta}_{j}^{SL}$ and is also called the active set of this estimator. Before stating our result, let us introduce some notations. For any
vector $a \in \mathbb{R}^{p}$ and any set of indexes $\mathcal{B}\in \{1,\ldots,p \}$, denote by $a_{\mathcal{B}}$ the restriction of the vector $a$ to the indexes in $\mathcal{B}$. In the same way, if we note $|\mathcal{B}|$ the cardinal of the set $\mathcal{B}$, then for any $s\times q$ matrix $M$, we use the following convention: i) $M_{\mathcal{B},\mathcal{B}}$ is the $|\mathcal{B}| \times |\mathcal{B}|$ matrix consisting of the lines and rows of $M$ whose indexes are in $\mathcal{B}$; ii) $M_{.,\mathcal{B}}$ is the $s \times |\mathcal{B}|$ matrix consisting of the rows of $M$ whose indexes are in $\mathcal{B}$; iii) $M_{\mathcal{B},.}$ is the $|\mathcal{B}| \times q$ matrix consisting of the lines of $M$ whose indexes are in $\mathcal{B}$. Moreover, we define $\widetilde{J}$ the $p\times p$ matrix $\mathbf{J}'\mathbf{J}$ where $\mathbf{J}$ was defined in \eqref{Jmatrix}. Finally we define for $j\in \{1,\ldots,p\}$, the quantity $\Omega_{j} = \Omega_{j}(\lambda,\mu,\mathcal{A}^{*},\beta^{*})$ by \begin{equation}\label{eq:OmegaCondi}
\Omega_{j}= \mathbf{C}_{j,\mathcal{A}^{*}}
(\mathbf{C}_{\mathcal{A}^{*},\mathcal{A}^{*}} + \mu
\widetilde{J}_{\mathcal{A}^{*},\mathcal{A}^{*}})^{-1} \left(2^{-1} \mathop{\rm
Sgn}(\beta_{\mathcal{A}^{*}}^{*}) + \frac{\mu}{\lambda}
\widetilde{J}_{\mathcal{A}^{*},\mathcal{A}^{*}}\beta_{\mathcal{A}^{*}}^{*}\right) - \frac{\mu}{\lambda} \widetilde{J}_{j,\mathcal{A}^{*}} \beta_{\mathcal{A}^{*}}^{*},
\end{equation}
where $\mathbf{C}$ is defined as in Theorem~\ref{thm:AsymptoticNormality}. Now consider the following conditions: {\it for every $j\in(\mathcal{A}^{*})^{c}$}
\begin{equation}\label{SufficientCondition}
|\Omega_{j}(\lambda,\mu,\mathcal{A}^{*},\beta^{*})| < 1,
\end{equation}
\begin{equation}\label{NecessaryCondition}
|\Omega_{j}(\lambda,\mu,\mathcal{A}^{*},\beta^{*})| \leq 1.
\end{equation}
These conditions on the correlation matrix $\mathbf{C}$ and the regression vector $\beta_{\mathcal{A}^{*}}^{*}$ are the analogues respectively of the sufficient and necessary conditions derived for the Lasso (\cite{Zou-Adaptive-Lasso}, \cite{Zhao-Yu-Consist-Lasso} and \cite{Yuan-Lin-Garrotte}). Now we state the consistency results \begin{thm}\label{thm:Sufficient-Var-consistant}
If condition~\eqref{SufficientCondition} holds, then for every couple of regularization parameters $(\lambda_{n},\mu_{n})$ such that $\lambda_{n} \rightarrow 0$, $\lambda_{n}n^{1/2}\rightarrow\infty$ and $\mu_{n} \rightarrow 0$, the S-Lasso estimator $\hat{\beta}^{SL}$ as defined in \eqref{eq:penalized-risk}-\eqref{critere_S-lasso} is consistent in variable selection. That is
\begin{equation*}
\mathbb{P}(\mathcal{A}_{n}=\mathcal{A}^{*}) \rightarrow 1, \quad
\text{when}\,\, n\rightarrow \infty.
\end{equation*}
\end{thm}
\begin{thm}\label{thm:Necessary-Var-consistant}
If there exist sequences $(\lambda_{n},\mu_{n})$ such that $\beta^{SL}$ converges to $\beta^*$ and $\mathcal{A}_n$ converges to $\mathcal{A}^*$ in probability, then condition~\eqref{NecessaryCondition} is satisfied.
\end{thm}
We just have established necessary and sufficient conditions to the selection consistency of the S-Lasso estimator. Due to the assumptions needed in Theorem~\ref{thm:Sufficient-Var-consistant} (more precisely $\lambda_{n} n^{1/2} \rightarrow \infty$), root-$n$ consistency and variable selection consistency cannot be treated here simultaneously. We may want to know if the S-Lasso estimator can be consistent with a slower rate than $n^{1/2}$ and consistent in variable selection in the same time.
\begin{rmq}
Here are special cases of conditions~\eqref{SufficientCondition}-~\eqref{NecessaryCondition}.\\
{\it When $\mu = 0$ and $\mu / \lambda = 0$}: these conditions are exactly the sufficient and necessary conditions of the
Lasso estimator. In this case Yuan and Lin \cite{Yuan-Lin-Garrotte} showed that the condition~\eqref{SufficientCondition} becomes necessary and sufficient for the Lasso estimator consistency in variable selection.\\
{\it When $\mu=0$ and $\mu / \lambda = \gamma \ne 0$}:
in this case, condition~\eqref{SufficientCondition} becomes
\begin{equation*}
\sup_{j\in(\mathcal{A}^{*})^{c}}|\mathbf{C}_{j,\mathcal{A}^{*}}
\mathbf{C}_{\mathcal{A}^{*},\mathcal{A}^{*}}^{-1}(2^{-1} \mathop{\rm
Sgn}(\beta_{\mathcal{A}^{*}}^{*}) + \gamma
\widetilde{J}_{\mathcal{A}^{*},\mathcal{A}^{*}} \beta_{\mathcal{A}^{*}}^{*}) -
\gamma \widetilde{J}_{j,\mathcal{A}^{*}} \beta_{\mathcal{A}^{*}}^{*}| < 1.
\end{equation*}
Here a good calibration of $\gamma$ leads to consistency in variable selection:
\begin{itemize}
\item if $(\mathbf{C}_{j,\mathcal{A}^{*}}
\mathbf{C}_{\mathcal{A}^{*},\mathcal{A}^{*}}^{-1}
\widetilde{J}_{\mathcal{A}^{*},\mathcal{A}^{*}} -
\widetilde{J}_{j,\mathcal{A}^{*}})\beta_{\mathcal{A}^{*}}^{*}
>0$, then $\gamma$ must be chosen between \\ $\displaystyle{-\frac{1 + 2^{-1} \mathbf{C}_{j,\mathcal{A}^{*}}
\mathbf{C}_{\mathcal{A}^{*},\mathcal{A}^{*}}^{-1}  \mathop{\rm
Sgn}(\beta_{\mathcal{A}^{*}}^{*})}{(\mathbf{C}_{j,\mathcal{A}^{*}}
\mathbf{C}_{\mathcal{A}^{*},\mathcal{A}^{*}}^{-1}
\widetilde{J}_{\mathcal{A}^{*},\mathcal{A}^{*}} -
\widetilde{J}_{j,\mathcal{A}^{*}})\beta_{\mathcal{A}^{*}}^{*}}}$ and  $\displaystyle{\frac{1 -
2^{-1} \mathbf{C}_{j,\mathcal{A}^{*}}
\mathbf{C}_{\mathcal{A}^{*},\mathcal{A}^{*}}^{-1}  \mathop{\rm
Sgn}(\beta_{\mathcal{A}^{*}}^{*})}{(\mathbf{C}_{j,\mathcal{A}^{*}}
\mathbf{C}_{\mathcal{A}^{*},\mathcal{A}^{*}}^{-1}
\widetilde{J}_{\mathcal{A}^{*},\mathcal{A}^{*}} -
\widetilde{J}_{j,\mathcal{A}^{*}})\beta_{\mathcal{A}^{*}}^{*}}}$.
\item if $(\mathbf{C}_{j,\mathcal{A}^{*}}
\mathbf{C}_{\mathcal{A}^{*},\mathcal{A}^{*}}^{-1}
\widetilde{J}_{\mathcal{A}^{*},\mathcal{A}^{*}} -
\widetilde{J}_{j,\mathcal{A}^{*}})\beta_{\mathcal{A}^{*}}^{*} < 0$, then $\gamma$ must be chosen between the same quantities but with inversion in their order.
\end{itemize}
{\it When $ \mu \ne 0$ and $\mu / \lambda = \gamma \ne 0$}: this case is similar to the previous. In addition, it allows to have another control on the condition through a calibration with $\mu$, so that condition~\eqref{SufficientCondition} can be satisfied with a better control.
\end{rmq}

We conclude that if we sacrifice the optimal rate of convergence (i.e. root-$n$ consistency), we are able through a proper choice of the tuning parameters $(\lambda_n ,\mu_n)$ to get consistency in variable selection. Note that Zou \cite{Zou-Adaptive-Lasso} showed that the Lasso estimator cannot be consistent in variable selection even with a slower rate of convergence than $\sqrt{n}$. He then added weights to the Lasso (i.e. the adaptive Lasso estimator) in order to get Oracles Properties (that is both asymptotic normality and variable selection consistency). Note that we can easily adapt techniques used in the adaptive Lasso to provide a weighted S-Lasso estimator which achieved the Oracles Properties.

\section{Theoretical results when dimension $p$ is larger than sample size $n$}\label{sec:theoryPgrd}
In this section, we propose to study the performance of the S-Lasso estimator in the high dimensional case. In particular, we provide a non-asymptotic bound on the squared risk. We also provide bound on the estimation risk under the sup-norm (i.e., the $l_{\infty}$-norm: $\|\hat{\beta}^{SL} - \beta^*\|_{\infty} = \sup_{j}|\hat{\beta}_j^{SL} - \beta_j^*|$). This last result helps us to provide a variable selection consistent estimator obtained through thresholding the S-Lasso estimator. The results of this section are proved in Appendix~B.
\subsection{Sparsity Inequality}\label{sec:SOI}
Now we establish a Sparsity Inequality (SI) achieved by the S-Lasso estimator, that is a bound on the squared risk that takes into account the sparsity of the oracle regression vector $\displaystyle{\beta^*}$. More precisely, we prove that the rate of convergence is $\displaystyle{|\mathcal{A}^*| \log (n) / n}$. For this purpose, we need some assumptions on the Gram matrix $\displaystyle{\mathbf{C}_n}$ which is normalized in our setting. Recall that $\xi_{j} = (x_{1,j},\ldots,x_{n,j})'$. Then we define the regularization parameters $\lambda_n$ and $\mu_n$ in the following forms:  
\begin{eqnarray}\label{eq:weight1}
\lambda_{n} = \kappa_1 \sigma \sqrt{\frac{\log(p)}{n}}, \quad \text{and} \quad \mu_{n} = \kappa_2 \sigma^2 \frac{\sqrt{\log(p)}}{n},
\end{eqnarray}
where $\kappa_1 >2\sqrt{2}$ and $\kappa_2$ is positive constants. Let us define the maximal correlation quantity $\rho_1 = \max_{j\in \mathcal{A}^*} \max_{ \underset{k \not = j}{k\in \{1,\ldots,p \} } } |(\mathbf{C}_n)_{j,k}|$. 
Using these notations, we formulate the following assumptions:
\begin{itemize}
\item {\it Assumption (A1). The true regression vector $\beta^{*}$ is such that there exists a finite constant $L_1$ such that:
\begin{equation}\label{Ass:normGamma}
\beta_{\mathcal{A}^*}^{*'} \widetilde{J}_{\mathcal{A}^{*},\mathcal{A}^{*}} \beta_{\mathcal{A}^*}^{*} \leq
L_1\log(p)\,|\mathcal{A}^*|,
\end{equation}
where $\widetilde{J} = \mathbf{J}'\mathbf{J}$ where $\mathbf{J}$ was defined in \eqref{Jmatrix}.}
\item {\it Assumption (A2). We have:
\begin{equation}\label{Ass:coherence}
\rho_1 \leq \frac{1}{16 |\mathcal{A^*}|}.
\end{equation}}
\end{itemize}

Note that Assumption~(A1) is not restrictive. A sufficient condition is that the larger non-zero component of $\beta_{\mathcal{A}^*}^{*}$ is bounded by $L_1\log(p)$ which can be very large. Assumption~(A2) is the well-known coherence condition considered in \cite{buneaTsyAggregaGauss}, which has been introduced in \cite{DonohoCoherence}. Most of SIs provided in the literature use such a condition. We refer to \cite{buneaTsyAggregaGauss} for more details.\\
Theorem~\ref{xasp} below provides an upper bound for the squared error of the estimator $\hat{\beta}^{SL}$ and for its $l_1$ estimation error which takes into account the sparsity index $|\mathcal{A}^*|$. 
\begin{thm}\label{xasp} Let us consider the linear regression model~\eqref{eq_depart}. Let $\hat{\beta}^{SL}$ be S-Lasso estimator. Let $\mathcal{A}^*$ be the sparsity set. Suppose that $p\geq n$ (and even $p\gg n$).
If Assumptions (A1)--(A2) hold, then with probability greater than
$1- u_{n,p}$, we have
\begin{equation}\label{bhu}
\| X\hat{\beta}^{SL} - X\beta ^* \|_{n}^{2}  \le c_2 \frac{\log (p) |\mathcal{A}^*|}{n},
\end{equation}
and
\begin{equation}\label{bhuu}
| \hat{\beta}^{SL} - \beta ^* |_{1}  \le c_1  \sqrt{\frac{\log
(p)}{n}} |\mathcal{A}^*|,
\end{equation}
where $ c_2 = (16 \kappa_1^2 + L_1 \kappa_2)\sigma^2$, $c_1 = (16 \kappa_1 +  L_1 \kappa_1^{-1} \kappa_2) \sigma$ and
where $u_{n,p}=p^{1-\kappa_1^2/8}$ with $\kappa_1$ and $\kappa_2$, the constants appearing in \eqref{eq:weight1}.
\end{thm}

The proof of Theorem~\ref{xasp} is based on the 'argmin' definition of  the estimator $\hat{\beta}^{SL}$ and some technical concentration inequalities. Similar bounds were provided for the Lasso estimator in \cite{SparLassBTW07}. Let us mention that the constants $c_1$ and $c_2$ are not optimal. We focused our attention on the dependency on $n$ (and then on $p$ and $|\mathcal{A}^*|$). It turns out that our results are near optimal. For instance, for the $l_2$ risk, the S-Lasso estimator reaches nearly the optimal rate $\frac{|\mathcal{A}^*|}{n}\, \log (\frac{p}{|\mathcal{A}^*|}+1)$ up to a logarithmic factor \cite[Theorem 5.1]{buneaTsyAggregaGauss}.

\subsection{Sup-norm bound and variable selection}\label{sec:supNorm}

Now we provide a bound on the sup-norm $\|\beta^{*} - \hat{\beta}^{SL}\|_{\infty}$. Thanks to this result, one may be able to define a rule in order to get a variable selection consistent estimator when $p \gg n$. That is, we can construct an estimator which succeeds to recover the support of $\beta^{*}$ in high dimensional settings.\\ Small modifications are to be imposed to provide our selection results in this section. Let $K_n$ be the symmetric $p \times p$ matrix defined by $K_n = \mathbf{C}_n + \mu_n \widetilde{J}$. Instead of Assumption~(A2), we will consider the following
\begin{itemize}
\item {\it Assumption~(A3). We assume that 
\begin{equation*}
\max_{ \underset{k \not = j}{j,\,k\in
\{1,\ldots,p \} }} |(K_n)_{j,k}| \leq \frac{1}{16 |\mathcal{A^*}|}.
\end{equation*}}
\end{itemize}

\begin{rmq}\label{rq:cooo}
Note that the matrix $\widetilde{J}$ is tridiagonal with its off-diagonal terms equal to $-1$. If we do not consider the diagonal terms, we remark that $\mathbf{C}_n$ and $K_n$ differ only in the terms on the second diagonals (i.e., $(K_n)_{j-1,j} \neq (\mathbf{C}_n)_{j-1,j}$ for $j=2,\ldots,p$ as soon as $\mu_n \neq 0$). Then, as we do not consider the diagonal terms in Assumptions~(A2) and~(A3), they differ only in the restriction they impose to terms on the second diagonals. Terms in the second diagonals of $\mathbf{C}_n$ correspond to correlations between successive covariates. Then when high correlations exist between successive covariates, a suitable choice of $\mu_n$ makes Assumption~(A3) satisfied while Assumptions~(A2) does not. Hence, Assumption~(A3) fits better with setup considered in the paper.
\end{rmq}

In the sequel, a convenient choice of the tuning parameter $\mu_n$ is $\mu_n = \kappa_3 \sigma /\sqrt{n\, \log{(p)}}$, where $\kappa_3 >0$ is a constant. Moreover, from Assumption~(A1), we have $\beta^{*'}_{\mathcal{A}^*}\widetilde{J}_{\mathcal{A}^{*},\mathcal{A}^{*}} \beta_{\mathcal{A}^*}^{*}\leq L_1\,\log{(p)} |\mathcal{A}^*|$. This inequality guarantees the existence of a constant $ L_2 > 0$ such that $\| \widetilde{J}\beta^{*} \|_{\infty} \leq L_2  \,\log{(p)}$.
\begin{thm}
\label{th:supNorm}
Let us consider the linear regression model~\eqref{eq_depart}. Let $\lambda_n = \kappa_1\sigma\sqrt{\log (p)/n}$ and $\mu_n = \kappa_3 \sigma /\sqrt{n\, \log{(p)}}$ with $\kappa_1>2\sqrt{2}$ and $\kappa_3 > 0$. Suppose that $p\geq n$ (and even $p\gg n$). Under Assumptions~(A1) and (A3) and with probability greater than $1- p^{1-\frac{\kappa_1^{2}}{8}}$, we have 
\begin{eqnarray*}
\| \hat{\beta}^{SL} - \beta^{*} \|_{\infty} \leq \tilde{c} \sqrt{\frac{\log{(p)}}{n}},
\end{eqnarray*}
where $\tilde{c}$ equals to
\begin{scriptsize}
\begin{eqnarray*}
\frac{1}{1 + \frac{B\sigma}{n}} \left( \frac{3}{4} + \frac{1}{\alpha-1} + \frac{4 L_1 B}{9 \alpha^2 A^2}    + \frac{2 L_1 B}{3 \alpha A^2} + \sqrt{ \frac{2 L_1 B}{3 \alpha(\alpha -1) A^2} +\frac{8 L_1 \,L_2 B^2}{9\alpha(\alpha-1) A^4} \lambda_n}  + (\frac{4 L_2 B}{3 A^2}+ \frac{L_2 B}{A^2}) \lambda_n \right).
\end{eqnarray*}
\end{scriptsize}
\end{thm}
Note that the leading term in $\tilde{c}$ is $\frac{3}{4} + \frac{1}{\alpha-1} + \frac{4 L_1 B}{9 \alpha^2 A^2} + \frac{2 L_1 B}{3 \alpha A^2} + \sqrt{ \frac{2 L_1 B}{3 \alpha(\alpha -1) A^2} }$. One may find back the result obtained for the Lasso by setting $L_1$ to zero \cite{KarimNormSup}. Secondly, the calibration of $\mu_n$ aims at making the convergence rate under the sup-norm equal to $\sqrt{\log{(p)}/n}$. On one hand, the proof of Theorem~\ref{th:supNorm} allows us to choose this parameter with a faster convergence to zero without affecting the rate of convergence. On the other hand, a more restrictive Assumption~(A1) on $\beta^{*'}_{\mathcal{A}^*}\widetilde{J}_{\mathcal{A}^{*},\mathcal{A}^{*}} \beta_{\mathcal{A}^*}^{*}$ and $\| \widetilde{J}\beta^{*} \|_{\infty}$ can be formulated in order to make $\mu_n$ converge slower to zero. If we let $\beta_{\mathcal{A}^*}^{*'} \widetilde{J}_{\mathcal{A}^{*},\mathcal{A}^{*}} \beta_{\mathcal{A}^*}^{*} \leq
L_1 \,|\mathcal{A}^*|$ in Assumption~(A1), we can set $\mu_n$ as $\mathcal{O}(\sqrt{\log{(p)}/n})$, the slower convergence we can get for $\mu_n$.\\
Let us now provide a consistent version of the S-Lasso estimator. Consider $\hat{\beta}^{ThSL}$, the thresholded S-Lasso estimator defined by $\hat{\beta}^{ThSL} = \hat{\beta}^{SL} \mathbb{I}(\hat{\beta}^{SL} \geq \tilde{c}\sqrt{\log{(p)}/{n}})$ where $\tilde{c}$ is given in Theorem~\ref{th:supNorm}. This estimator consists of the S-Lasso estimator with its small coefficients reduced to zero. We then enforce the selection property of the S-Lasso estimator. Variable selection consistency of this estimator is established under one more restriction:
\begin{itemize}
\item {\it Assumption (A4). The smallest non-zero coefficient of $\beta^*$ is such that there exists a constant $c_{l} > 0$ with
\begin{equation*}
\min_{j\in\mathcal{A}^*}|\beta_{j}^{*}| >  c_l \sqrt{\frac{\log{(p)}}{n}}.
\end{equation*}}
\end{itemize}
Assumption~(A4) bounds from below the smallest regression coefficient in $\beta^*$. This is a common assumption to provide sign consistency in the high dimensional case. This condition appears in \cite{MeinYuSelect,WainSelection,ZhangHuangConsist,Zhao-Yu-Consist-Lasso} but with a larger (in term of sample size $n$ dependence) and then more restrictive threshold. We refer to \cite{KarimNormSup} for a longer discussion. An equivalent lower bound in the oracle regression coefficients can be found in \cite{Bunea_consist,KarimNormSup}. With this new assumption, we can state the following sign consistency result.
\begin{thm}
\label{th:signConst_q2}
Let us consider the thresholded S-Lasso estimator $\hat{\beta}^{ThSL}$ as described above. Choose moreover $\lambda_n=\kappa_1\sigma\sqrt{\log (p)/n}$ and $\mu_n = \kappa_3\sigma/\sqrt{n\log{(p)}}$ with the positive constants $\kappa_1 > 2\sqrt{2}$ and $\kappa_3$. Under Assumptions~(A1), (A3) and (A4), if $c_l > 2\tilde{c}$ with $\tilde{c}$ is given by Theorem~\ref{th:supNorm}, with probability greater than $1- p^{1-\frac{\kappa_1^{2}}{8}}$, we have
\begin{equation}
\mathop{\rm Sgn}(\hat{\beta}^{ThSL}) = \mathop{\rm Sgn}(\beta^{*}),
\end{equation}
and then as $n\rightarrow +\infty$
\begin{equation}
\mathbb{P}(\mathop{\rm Sgn}(\hat{\beta}^{ThSL}) = \mathop{\rm Sgn}(\beta^{*})) \rightarrow 1.
\end{equation}
\end{thm}

\begin{rmq}\label{cor:ConsistThres}
As observed in Remark~\ref{rq:cooo}, Assumption~(A3) is more easily satisfied when correlation exists between successive covariates. Then in situations where the correlation matrix $\mathbf{C}_n$ is tridiagonal with its off-diagonal terms equal to $\delta$ with $\delta \in [0,1]$, the constant $\kappa_3$ appearing in the definition of $\mu_n$ can be adjusted in order to get Assumption~(A3) satisfied.
\end{rmq}

\section{Model Selection}\label{sec-Selection}
As already said [Remark~\ref{rq:RegulPath} in Section~\ref{solve}], each step
of the S-Lasso version of the LARS algorithm provides an estimator of
$\beta^{*}$. In this section, we are interested in the choice of the best
estimator according to its prediction accuracy. For a new $n\times p$ matrix
$x_{new}$ of instances (independent of $X$), denote $\hat{y}^{SL} =
\hat{\beta}^{SL}x_{new}$ the estimator of its unknown response value $y_{new}$
and $m = \mathbb{E}(y_{new}|x_{new})$. We aim to minimize the true risk
$\mathbb{E}\left\{\|m-\hat{y}^{SL} \|_{n}^{2}\right\}$. First, we easily obtain
\begin{equation*} \mathbb{E}\left\{\|m-\hat{y}^{SL} \|_{n}^{2}\right\}=
\mathbb{E}\{\|Y-\hat{y}^{SL} \|_{n}^{2} -  \sigma^{2}+
2n^{-1}\sum_{i=1}^{n}\mathop{\rm Cov}(y_{i},\hat{y}_{i}^{SL})\},
\end{equation*}
where the expectation is taken over the random variable $Y$. The last term in
this equation was called \textit{optimism} \cite{Efron-Optimism}.
Moreover, Tibshirani \cite{Tibshirani-LASSO} links this quantity to the
\textit{degree of freedom} $\mathop{\rm df}(\hat{y}^{SL})$ of the estimator
$\hat{y}^{SL}$, so that the above equality becomes
\begin{equation}\label{CritereSelection}
\mathbb{E}\left\{\|m-\hat{y}^{SL} \|_{n}^{2}\right\}=
\mathbb{E}\left\{\|Y-\hat{y}^{SL}\|_{n}^{2} - \sigma^{2} + 2 n^{-1}
\mathop{\rm df}(\hat{y}^{SL})\sigma^{2}\right\}.
\end{equation}
This final expression involves the degree of freedom which is unknown. Various
methods exist to estimate the degree of freedom as bootstrap \cite{Efron-Bootstrap}
or data perturbation methods \cite{ShenYeSelection}. We give an
explicit form to the degree of freedom in order to reduce the
computational cost as in \cite{Efron-LARS} and \cite{Zou-df-Lasso}.
\\
\\
\textbf{Degrees of freedom:} the degree of freedom is a
quantity of interest in model selection. Before stating our result, let us
introduce some useful properties about the regularization paths of the S-Lasso
estimator:\\
Given a response $Y$, and a regularization parameter $\mu\geq0$, there is a
finite sequence $0=\lambda^{(K)}<\lambda^{(K-1)}<\ldots<\lambda^{(0)}$ such
that $\hat{\beta}^{SL}=\mathbf{0}$ for every $\lambda \geq \lambda^{(0)}$. In
this
notation, superscripts correspond to the steps of the S-Lasso version of the LARS algorithm.\\
Given a response $Y$, and a regularization parameter $\mu\geq0$, for
$\lambda\in(\lambda^{(k+1)},\lambda^{(k)})$, the same covariates are used to
construct the estimator. Let us note $\mathcal{A}_{\zeta}$ the active set for a
fixed couple $\zeta=(\lambda,\mu)$ and $X_{.,\mathcal{A}_{\zeta}}$ the
corresponding design matrix.

In what follows, we will use the subscript $\zeta$ to emphasize the fact that
the considered quantity depends on $\zeta$.
\begin{thm}\label{th:dll_exact}
For fixed $\mu\geq0$ and $\lambda>0$, an unbiased estimate of the effective
degree of freedom of the S-Lasso estimate is given by
\begin{equation*}
\widehat{\mathop{\rm df}}(\hat{y}_{\zeta}^{SL})=\mathop{\rm
Tr}\left[X_{.,\mathcal{A}_{\zeta}}
\left(X_{.,\mathcal{A}_{\zeta}}'X_{.,\mathcal{A}_{\zeta}} +
\mu\widetilde{J}_{\mathcal{A}_{\zeta},\mathcal{A}_{\zeta}}\right)^{-1}
X_{.,\mathcal{A}_{\zeta}}'\right],
\end{equation*}
where $\widetilde{J}=\mathbf{J}'\mathbf{J}$ is defined by
\begin{equation}\label{matriceddf}
\widetilde{J}=\begin{pmatrix}
  1 & -1 & 0 & \ldots & 0 \\
  -1 & 2 & -1 & \ddots & \vdots \\
  0 & \ddots & \ddots & \ddots & 0 \\
  \vdots & \ddots & -1 & 2 & -1 \\
  0 & \ldots & 0 & -1 & 1
\end{pmatrix}.
\end{equation}
\end{thm}

As the estimation given in Theorem~\ref{th:dll_exact} has an important
computational cost, we propose the following estimator of the degree of freedom
of the S-Lasso estimator:
\begin{equation}\label{eq:ddl_S_Lasso}
\widehat{\mathop{\rm df}}(\hat{y}_{\zeta}^{SL}) = \frac{|\mathcal{A}_{\zeta}| -
2}{1 + 2\mu} + \frac{2}{1 + \mu},
\end{equation}
which is very easy to compute. Let $\mathbf{I}_{s}$ be the $s \times s $
identity matrix where $s$ is an integer. We found the former approximation of
the degree of freedom under the orthogonal covariance matrix assumption (that
is $n^{-1}X'X = \mathbf{I}_{p}$). Moreover we approximate the matrix
$(\mathbf{I}_{|\mathcal{A}_{\lambda}|} +
\mu\widetilde{J}_{\mathcal{A}_{\lambda},\mathcal{A}_{\lambda}})$ by the
diagonal matrix with $1+\mu$ in the first and the last terms, and
$1+2\mu$ in the others.

\begin{rmq} [{\it Comparison to the Lasso and the Elastic-Net}]
A similar work leads to an estimation of the degree of freedom of the Lasso:
$\widehat{\mathop{\rm df}}(\hat{y}_{\zeta}^{L}) = |\mathcal{A}_{\zeta}|$ and to
an estimation of the degree of freedom of the Elastic-Net estimator:
$\widehat{\mathop{\rm df}}(\hat{y}_{\zeta}^{EN}) = |\mathcal{A}_{\zeta}| / (1 +
\mu)$. These approximations of the degrees of freedom provide the following
comparison for a fixed $\zeta$: $\widehat{\mathop{\rm
df}}(\hat{y}_{\zeta}^{SL}) \leq \widehat{\mathop{\rm df}}(\hat{y}_{\zeta}^{EN})
\leq \widehat{\mathop{\rm df}}(\hat{y}_{\zeta}^{L})$. A conclusion is that the
S-Lasso estimator is the one which penalizes the smaller models, and the Lasso
estimator the larger. As a consequence, the S-Lasso estimator should select
larger models than the Lasso or the Elastic-Net estimator.
\end{rmq}

\section{The Normalized S-Lasso estimator}\label{sec-Norm}
In this section, we look for a scaled S-Lasso estimator which would have better
empirical performance than the original S-Lasso presented above. The idea
behind this study is to better control shrinkage. Indeed, using the S-Lasso
procedure \eqref{eq:penalized-risk}-\eqref{critere_S-lasso} induces double
shrinkage: one using the Lasso penalty and the other using the fusion penalty.
We want to undo the shrinkage implied by the fusion penalty as shrinkage is
already ensured by the Lasso penalty. We then suggest to study the S-Lasso
criterion \eqref{eq:penalized-risk}-\eqref{critere_S-lasso} without the Lasso
penalty (i.e. with only the $l_{2}$-fusion penalty) in order to find the constant we have to
scale with.
\\
\\
Define
\begin{equation*}
\tilde\beta=\underset{\beta \in \mathbb{R}^{p}}{\argmin}\Vert Y-
X\beta\Vert_n^2  + \mu\sum_{j=2}^{p} \left(\beta_{j}-\beta_{j-1}\right)^{2}.
\end{equation*}
We easily obtain $\tilde\beta=((X'X)/n + \mu\widetilde{J})^{-1}(X'Y)/n
:=\mathbf{L}^{-1}(X'Y)/n$ where $\widetilde{J}$ is given by \eqref{matriceddf}.
Moreover as the design matrix $X$ is standardized, the symmetric
matrix~$\mathbf{L}$ can be written
\begin{equation*}
\mathbf{L} = \begin{pmatrix}
  1+\mu & \frac{\xi_{1}' \, \xi_{2}}{n} -\mu & \frac{\xi_{1}' \, \xi_{3}}{n}& \ldots & \frac{\xi_{1}' \, \xi_{p}}{n} \\
   & 1+2\mu & n^{-1}\xi_{2}' \xi_{3}-\mu & \ldots & \vdots \\
   &  & \ddots & \ddots & \frac{\xi_{p-2}' \, \xi_{p}}{n} \\
   &  &  & 1+2\mu & \frac{\xi_{p-1}' \, \xi_{p}}{n}-\mu \\
   &  &  &  & 1+\mu
\end{pmatrix}.
\end{equation*}
In order to get rid of the shrinkage due to the fusion penalty, we force
$\mathbf{L}$ to have ones (or close to a diagonal of ones) in its diagonal
elements. Then we scale the estimator $\tilde\beta$ by a factor $c$. Here are
two choice we will use in the following of the paper: i) the first is $c=1+\mu$
so that the first and the last diagonal elements of $\mathbf{L}^{-1}$ become
equal to one; ii) the second is $c=1+2\mu$ which offers the advantage that all
the diagonal elements of $\mathbf{L}^{-1}$ become equal to one except the first
and the last. This second choice seems to be more appropriate to undo this
extra shrinkage and specially in high dimensional problem.
\\
\\
We first give a generalization of Lemma~\ref{equivalenceLasso}.
\begin{lm}
Given the dataset $(X,Y)$ and $(\lambda_{1},\mu)$. Define the augmented
dataset $(\widetilde{X},\widetilde{Y})$ by
$$\widetilde{X}=\nu_{1}^{-1} \begin{pmatrix}
X \\
\sqrt{n \mu}\mathbf{J}
\end{pmatrix} \quad
\text{and} \quad \widetilde{Y}=
\begin{pmatrix}
Y \\
\mathbf{0}
\end{pmatrix},$$ where $\nu_{1}$ is a constant which depends only on
$\mu$ and $\mathbf{J}$ is given by \eqref{Jmatrix}. Let $r=\lambda / \nu_1$ and
$b= (\nu_2/c) \beta$ where $\nu_{2}$ is a constant which depends only on $\mu$,
and $c$ is the scaling constant which appears in the previous study. Then the
S-Lasso criterion can be written
\begin{equation}\label{eq:BigVarCriter}
\left\|\widetilde{Y}-\widetilde{X} b\right\|_{n}^{2}+ r | b |_{1}.
\end{equation}
Let $\hat{b}$ be the minimizer of this Lasso-criterion, then we define the
Scaled Smooth Lasso (SS-Lasso) by
\begin{equation*}
\hat{\beta}^{SSL} = \hat{\beta}^{SSL}(\nu_1 ,\nu_2, c) =(c/\nu_2)\,\hat{b}.
\end{equation*}
Moreover, let $\widetilde{J}=\mathbf{J}'\mathbf{J}$. Then we have
\begin{equation}\label{eq:NSLassoCriter}
\hat{\beta}^{SSL}= \underset{\beta \in \mathbb{R}^{p}}{\argmin}\left\{
\frac{\nu_2}{\nu_1}\beta' \left(\frac{\frac{X'X}{n} + \mu \widetilde{J}}{c}
\right) \beta - 2 \frac{Y'X}{n}\beta + \lambda \sum_{j=1}^{p} |\beta_j|
\right\}.
\end{equation}
\end{lm}
Equation~\eqref{eq:NSLassoCriter} is only a rearrangement of the Lasso
criterion~\eqref{eq:BigVarCriter}. The SS-Lasso
expression~\eqref{eq:NSLassoCriter} emphasizes the importance of the scaling
constant $c$. In a way, the SS-Lasso estimator stabilizes the Lasso estimator
$\hat{\beta}^{L}$ (criterion~\eqref{eq:BigVarCriter} based in $(X,Y)$ instead
of $(\widetilde{X},\widetilde{Y})$) as we have
\begin{equation*}
\hat{\beta}^{L}= \underset{\beta \in \mathbb{R}^{p}}{\argmin}\left\{ \beta'
\left(\frac{X'X}{n}\right) \beta - 2 \frac{Y'X}{n}\beta + \lambda
\sum_{j=1}^{p} |\beta_j| \right\}.
\end{equation*}

The choice of $\nu_1$ and $\nu_2$ should be linked to this scaling constant $c$
in order to get better empirical performances and to have less parameters to
calibrate. Let us define some specific cases. i) {\it Case 1: When $\nu_1 =
\nu_2 = \sqrt{1+ \mu}$ and $c=1$}: this is the "original" S-Lasso
estimator as seen in Section~\ref{solve}. ii) {\it Case 2: When $\nu_1 = \nu_2
= \sqrt{1+ \mu}$ and $c=1+\mu$}: we call this scaled S-Lasso
estimator Normalized Smooth Lasso (NS-Lasso) and we note it
$\hat{\beta}^{NSL}$. In this case, we have $\hat{\beta}^{NSL} = (1+\mu
\hat{\beta}^{SL})$. iii) {\it Case 3: When $\nu_1 = \nu_2 = \sqrt{1+
2\mu}$ and $c=1+ 2 \mu$}: we call this scaled version Highly
Normalized Smooth Lasso (HS-Lasso) and we note it $\hat{\beta}^{HSL}$.

Others choices are possible for $\nu_1$ and $\nu_2$ in order to better control
shrinkage. For instance we can consider a compromise between the NS-Lasso and
the HS-Lasso by defining $\nu_1 = 1 + \mu$ and $ \nu_2 = 1+2\mu $.

\begin{rmq}[{\it Connection with Soft Thresholding}]\label{sec-Soft} Let us consider the limit case of the
NS-Lasso estimator. Note $\hat{\beta}_{\infty}^{NSL}=\lim_{\mu \rightarrow
\infty} \hat{\beta}^{NSL}$, then using \eqref{eq:NSLassoCriter}, we have
\begin{equation*}
\hat{\beta}_{\infty}^{NSL}=\argmin_{\beta} \{ \beta'\beta - 2 Y' X\beta +
\lambda |\beta |_{1}\}.
\end{equation*}
As a consequence, $(\hat{\beta}_{\infty}^{NSL})_j = \left(|Y'\xi_{j}| -
\frac{\lambda}{2} \right)_{+} \mathop{\rm Sgn}(Y'\xi_{j})$ which is the
Univariate Soft Thresholding \cite{DonohoSoftThresh}. Hence, when $\mu
\rightarrow \infty$, the NS-Lasso works as if all the covariates were
independent. The Lasso, which corresponds to the NS-Lasso when $\mu= 0$,
often fails to select covariates when high correlations exist between relevant
and irrelevant covariates. It seems that the NS-Lasso is able to avoid such
problem by increasing $\mu$ and working as if all the covariates were
independent. Then for a fixed $\lambda$, the control of the regularization
parameter $\mu$ appears to be crucial. When we vary it, the NS-Lasso
bridges the Lasso and the Soft Thresholding.
\end{rmq}

\section{Extension and comparison}\label{sec:Ext}
All results obtained in the present paper can be generalized to all penalized least square estimators for which the penalty term can be written as:
\begin{equation}\label{penGen}
\pen(\beta) = \lambda | \beta |_{1} + \beta' M \beta,
\end{equation}
where $M$ is $p \times p$ matrix. In particular, our study can be extended for instance to the Elastic-Net estimator with the special choice $M= \mathbf{I}_{p}$. Such an observation underlines the superiority of the S-Lasso estimator on the Elastic-Net in some situations. Indeed, let us consider the variable selection consistency in the high dimensional setting (cf. Section~\ref{sec:supNorm}). Regarding the Elastic-Net, Assumption~(A3) becomes
\begin{itemize}
\item {\it Assumption~(A3-EN). We assume that 
\begin{equation}\label{eq:ENassump}
\max_{ \underset{k \not = j}{j,\,k\in
\{1,\ldots,p \} }} |(\mathbf{C}_n)_{j,k} + \mu_n \mathbf{I}_p| \leq \frac{1}{16 |\mathcal{A^*}|}.
\end{equation}}
\end{itemize}
Since the identity matrix is diagonal and since the maximum in \eqref{eq:ENassump} is taken over indexes $k\neq j$, condition~\eqref{eq:ENassump} reduces to $\max_{ \underset{k \not = j}{j,\,k\in
\{1,\ldots,p \} }} |(\mathbf{C}_n)_{j,k}| \leq \frac{1}{16 |\mathcal{A^*}|}$. This makes Assumption~(A3-EN) similar to the assumption needed to get the variable selection consistency of the Lasso estimator \cite{Bunea_consist}. Hence, we get no gain to use the Elastic-Net in a variable selection consistency point of view in our framework. This ables us to think that the S-Lasso outperforms the Elastic-Net at least on examples as the one in Remark~\ref{cor:ConsistThres}. Recently, Jia and Yu~\cite{ElasticNetConsist08} studied the variable selection consistency of the Elastic-Net under an assumption called {\it Elastic Irrepresentable Condition}:
\begin{itemize}
\item {\it (EIC). There exists a positive constant $\theta$ such that for any $j\in(\mathcal{A}^*)^c$
\begin{equation*}
|\mathbf{C}_{j,\mathcal{A}^{*}} (\mathbf{C}_{\mathcal{A}^{*},\mathcal{A}^{*}} + \mu
\mathbf{I}_{\mathcal{A}^{*}})^{-1} \left(2^{-1} \mathop{\rm
Sgn}(\beta_{\mathcal{A}^{*}}^{*}) + \frac{\mu}{\lambda}
\beta_{\mathcal{A}^{*}}^{*}\right)| \leq 1-\theta.
\end{equation*}
}
\end{itemize}
This condition can be seen as a generalization of the {\it Irrepresentable Condition} involved in the Lasso variable selection consistency.\\
Let us discuss how the two assumptions can be compared in the case $p\gg n$. First, note that Assumption~(A3-EN), as well as EIC suggests low correlations between covariates. Moreover Assumption~(A1),~(A4) and~(A3-EN) seem more restrictive than EIC as all the correlations are constrained in~\eqref{eq:ENassump}. However, EIC is harder to interpret in term of the coefficients of the regression vector $\beta^*$. It also depends on the sign of $\beta^*$. The main difference is that the consistency result in the present paper holds uniformly on the solutions of the Elastic Net criterion while the result from~\cite{ElasticNetConsist08} hinges upon the existence of a consistent solution for variable selection. Obviously, this is more restrictive as we are certain to provide the sign-consistent solution under the EIC. Finally, we have also provided results on the sup-norm and sparsity inequalities on the squared risk of our estimators. Such results are new for estimators defined with the penalty~\eqref{penGen}, including the S-Lasso and the Elastic-Net.

\section{Experimental results}\label{sec-exper}
In the present section we illustrate the good prediction and selection
properties of the NS-Lasso and the HS-Lasso estimators. For this purpose, we
compare it to the Lasso and the Elastic-Net. It appears that S-Lasso is a good
challenger to the Elastic-Net \cite{Zou-E-Net} even when large correlations
between covariates exist. We further show that in most cases, our procedure
outperforms the Elastic-Net and the Lasso when we consider the ratio between
the relevant selected covariates and irrelevant selected covariates.
\\
\\
\textbf{Simulations:}
\\
\textit{Data.} Four simulations are generated according to the linear
regression model
\begin{equation*}
y=x\beta^{*} + \sigma\varepsilon, \quad \varepsilon\sim\mathcal{N}(0,1),\,\,
x=(\xi_1,\ldots,\xi_p)\in\mathbb{R}^p.
\end{equation*}
The first and the second examples were introduced in the original Lasso paper
\cite{Tibshirani-LASSO}. The third simulation creates a grouped covariates
situation. It was introduced in \cite{Zou-E-Net} and aims to point the
efficiency of the Elastic-Net compared to the Lasso. The last simulation
introduces large correlation between successive covariates.\\
\\
\begin{tabular}{p{0.3cm}p{14.1cm}}
(a) & In this example, we simulate $20$ observations with $8$ covariates. The
true regression vector is $\beta^{*}=(3,1.5,0,0,2,0,0,0)'$ so that only three
covariates are truly relevant. Let $\sigma=3$ and the correlation between
$\xi_{j}$ and $\xi_{k}$ such that $\mathop{\rm Cov}(\xi_j, \xi_k)=2^{-\abs{j-k}}$. \\ & \\
(b) & The second example is the same as the first one, except that we generate $50$ observations and that
$\beta_{j}^{*}=0.85$ for every $j\in\{1,\ldots,8\}$ so that all the covariates are relevant.
\\ & \\
(c) & In the third example, we simulate $50$ data with $40$ covariates. The true
regression vector is such that $\beta_{j}^{*}=3$ for $j=1,\ldots,15$
 and $\beta_{j}^{*}=0$ for $j=16,\ldots,40$. Let $\sigma=15$ and the covariates generated as follows:
\begin{eqnarray*}
\xi_{j}=Z_{1}+\varepsilon_{j}, & Z_{1}\sim\mathcal{N}(0,1), &  \quad j=1,\ldots,5, \\
\xi_{j}=Z_{2}+\varepsilon_{j}, & Z_{2}\sim\mathcal{N}(0,1), &\quad j=6,\ldots,10, \\
\xi_{j}=Z_{3}+\varepsilon_{j}, & Z_{3}\sim\mathcal{N}(0,1), &\quad
j=11,\ldots,15,
\end{eqnarray*}
where $\varepsilon_{j},\,j=1,\ldots,15$, are i.i.d. $\mathcal{N}(0,0.01)$
variables. Moreover for $j=16,\ldots,40$, the $\xi_{j}$'s are i.i.d
$\mathcal{N}(0,1)$ variables.
\end{tabular}
\begin{tabular}{p{0.3cm}p{14.1cm}}
(d) & In the last example, we generate $50$ data with $30$ covariates. The true
regression vector is such that
\begin{eqnarray*}
\beta_{j} = & 3-0.1j & \quad j=1,\ldots,10, \\
\beta_{j} = &-5+0.3j & \quad j=20,\ldots,25, \\
\beta_{j} = & 0 &  \quad \text{for the others } j.
\end{eqnarray*}
The noise is such that $\sigma=9$ and the correlations are such that
$\mathop{\rm Cov}(\xi_j, \xi_k)=\exp{(-\frac{\abs{j-k}}{2})}$ for $(j,k)\in\{11,\ldots,25\}^{2}$ and
the others covariates are i.i.d. $\mathcal{N}(0,1)$, also independent from
$\xi_{11},\ldots,\xi_{25}$. In this model there are big correlation between relevant covariates and even between relevant and irrelevant covariates.
\end{tabular}\\
\\
\textit{Validation.} The selection of the tuning parameters $\lambda$
and $\mu$ is based on the minimization of a $\mathop{\rm BIC}$-type criterion
\cite{Schwar-BIC}. For a given $\hat{\beta}$ the associated $\mathop{\rm
BIC}$~error is defined as:
\begin{equation*}
\mathop{\rm BIC}(\hat{\beta}) = \Vert Y- X\hat{\beta}\Vert_n^2 +
\frac{\log(n)\sigma^{2}}{n}\widehat{\mathop{\rm df}}(\hat{\beta}),
\end{equation*}
where $\widehat{\mathop{\rm df}}(\hat{\beta})$ is given by
\eqref{eq:ddl_S_Lasso} if we consider the S-Lasso and denotes its analogous
quantities if we consider the Lasso or the Elastic-Net. Such a criterion provides
an accurate estimator which enjoys good variable selection properties
(\cite{Shao-Critere-AIC-BIC} and \cite{Yang-ComparaisonAIC-BIC}). In simulation
studies, for each replication, we also provide the Mean Square Error (MSE) of
the selected estimator on a new and independent
dataset with the same size as training set (that is $n$).  This gives an information on the robustness of the procedures.\\
\\
\textit{Interpretations.} All the results exposed here are based on $200$
replications. Figure~\ref{fig:BIC} and Figure~\ref{fig:Test} give respectively
the $\mathop{\rm BIC}$~error and the test~error of the considered procedures in
each example. According to the selection part, Figure~\ref{fig:PlotA} shows the
frequencies of selection of each covariate for all the procedures, and
Table~\ref{NoN-Zero} shows the mean of the number of non-zeros coefficients
that each procedure selected. Finally for each procedure, Table~\ref{ratio}
gives the ratio between the number of relevant covariates and the number of
noise covariates that the procedures selected. Let us call $\mathop{\rm SNR}$ this ratio. Then we can express this ratio
as $$\mathop{\rm SNR} = \frac{\sum_{j\in\mathcal{A}_{n}}\mathbb{I}(j\in\mathcal{A}^{*})}
{\sum_{j\in\mathcal{A}_{n}}\mathbb{I}(j\notin\mathcal{A}^{*})}.$$This is a good
indication of the selection power of the procedures.

\begin{table}[t] 
\begin{center}
\begin{tabular}{|c|c|c|c|c|}
  \hline
  \text{Method} & \text{Example (a)} & \text{Example (b)} & \text{Example (c)} & \text{Example (d)} \\ \hline
  \text{Lasso} & 3.8\, [$\pm$0.1] & 6.5 \,[$\pm$0.1] & 6\, [$\pm$0.1] & 18.4\, [$\pm$0.2] \\
  \text{E-Net} & 4.9\, [$\pm$0.1] & 6.9\, [$\pm$0.1] & 15.9\, [$\pm$0.1]  & 20.5\, [$\pm$0.2] \\
  \text{NS-Lasso} & 3.9\, [$\pm$0.1] & 6.5\, [$\pm$0.1] & 15.3\, [$\pm$0.2] & 18.9\, [$\pm$0.2] \\
  \text{HS-Lasso} & 3.5\, [$\pm$0.1] & 5.9\, [$\pm$0.1] & 15\, [$\pm$0.1] & 18.1\, [$\pm$0.2] \\ \hline
\end{tabular}
\caption{Mean of the number of non-zero coefficients [and its standard
error] selected respectively by the Lasso, the Elastic-Net (E-Net), the
Normalized Smooth Lasso (NS-Lasso) and the Highly Smooth Lasso (HS-Lasso)
procedures.} \label{NoN-Zero}
\end{center}
\end{table}

\begin{table}[t]
\begin{center}
\begin{tabular}{|c|c|c|c|}
  \hline
  \text{Method} & \text{Example (a)} & \text{Example (c)} & \text{Example (d)} \\ \hline
  \text{Lasso} & 2.3\, [$\pm$0.1] & 2.9\, [$\pm$0.1] & 4.7\, [$\pm$0.2] \\
  \text{E-Net} & 1.7\, [$\pm$0.1] & 13.1\, [$\pm$0.3] & 3.4\, [$\pm$0.2] \\
  \text{NS-Lasso} & 2.5\, [$\pm$0.1] & 13.5\, [$\pm$0.3] & 6.8\, [$\pm$0.3] \\
  \text{HS-Lasso} & 1.79\, [$\pm$0.1] & 11.4\, [$\pm$0.3] & 6.4\, [$\pm$0.3] \\ \hline
\end{tabular}
\caption{Mean of the ratio between the number of relevant covariates and the
number of noise covariates ($\mathop{\rm SNR}$) [and its
standard error] that each of the Lasso, the Elastic-Net, the NS-Lasso and
the HS-Lasso procedures selected.} \label{ratio}
\end{center}
\end{table}

\begin{figure} [htb]
\begin{center}
\includegraphics[height=4in,width=5in] {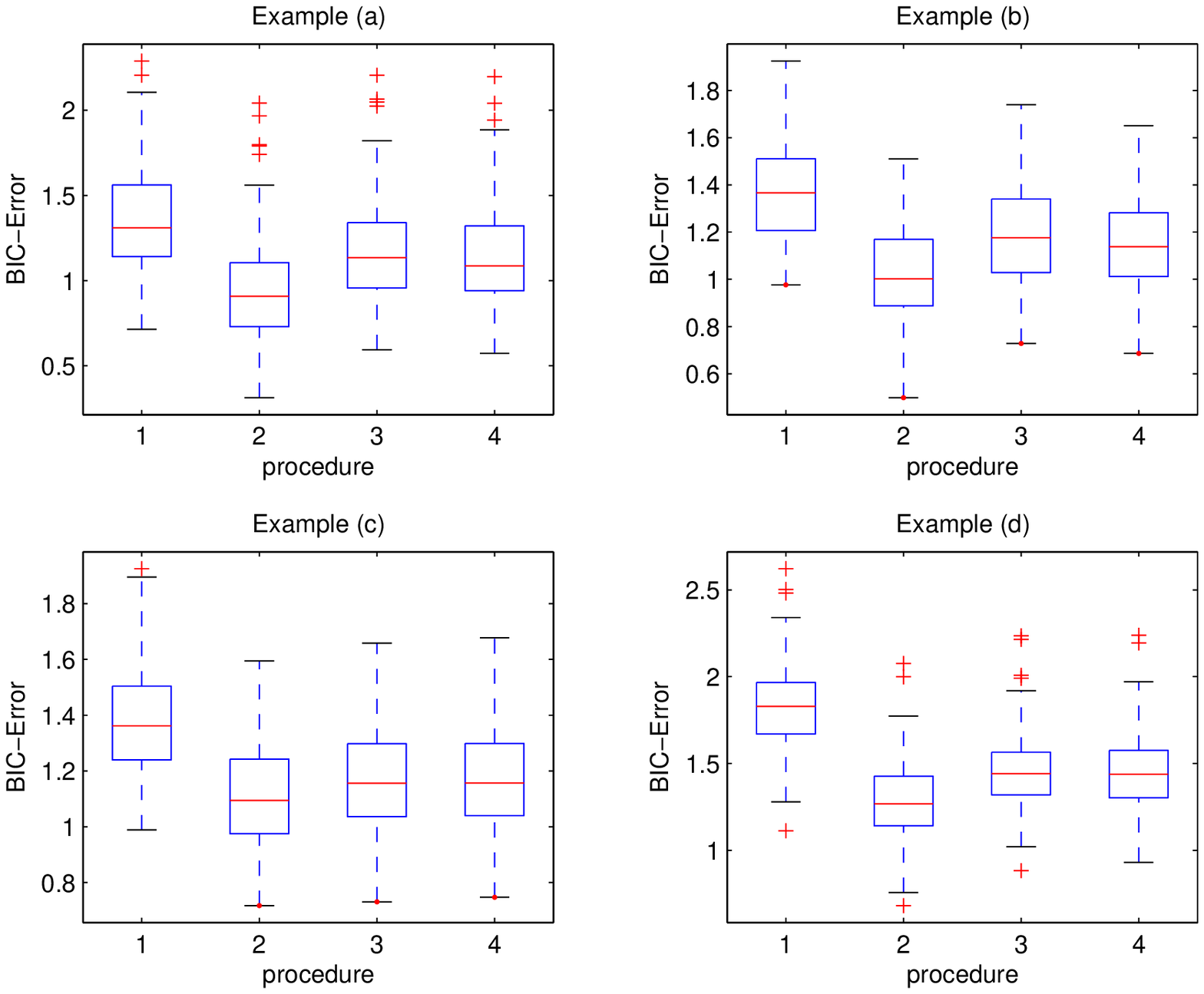}
\caption{$\mathop{\rm BIC}$~error in each example. For each plot, we construct
the boxplot for the procedure 1 = Lasso; 2 = Elastic-Net; 3 = NS-Lasso; 4 =
HS-Lasso} \label{fig:BIC}
\end{center}
\end{figure}

\begin{figure} [htb]
\begin{center}
\includegraphics[height=4in,width=5in] {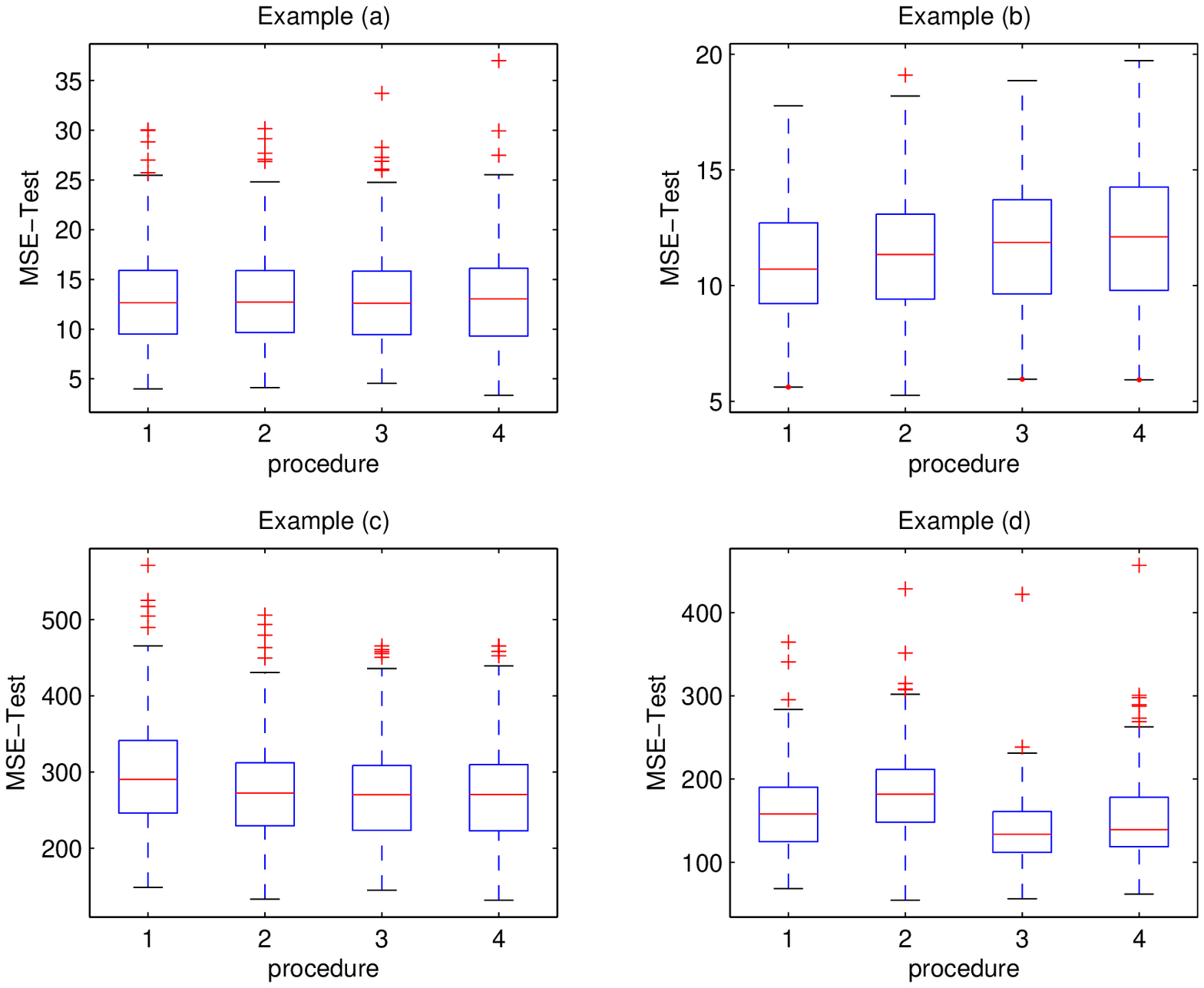}
\caption{Test~Error in each example. For each plot, we construct the boxplot
for the procedure 1 = Lasso; 2 = Elastic-Net; 3 = NS-Lasso; 4 = HS-Lasso}
\label{fig:Test}
\end{center}
\end{figure}

\begin{figure}
\begin{minipage}[t]{0.20\textwidth}
\begin{center}
\includegraphics[height=1.3in,width=1.3in] {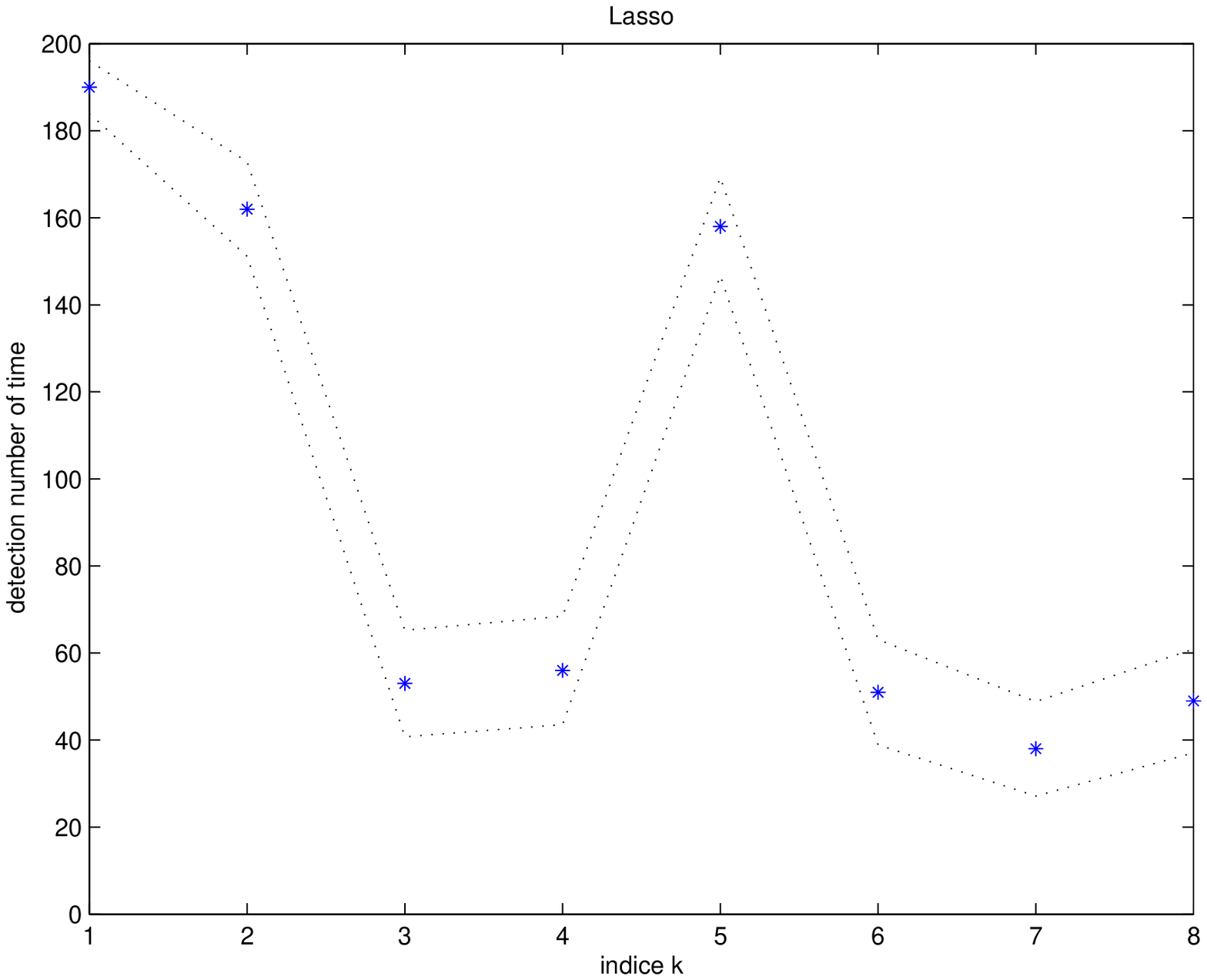}
\includegraphics[height=1.3in,width=1.3in] {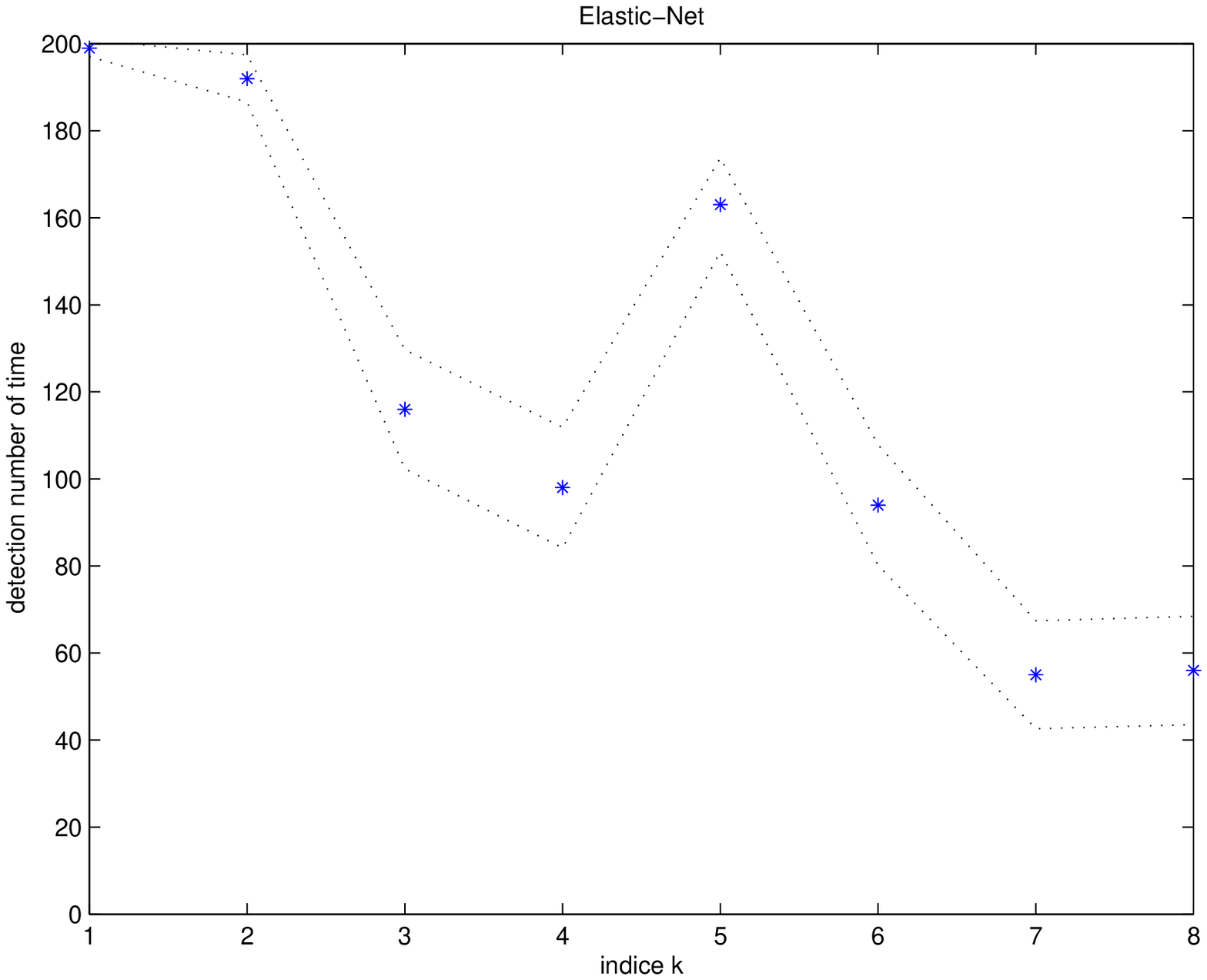}
\end{center}
\end{minipage}
\hfill\begin{minipage}[t]{0.20\textwidth} 
\begin{center}
\includegraphics[height=1.3in,width=1.3in] {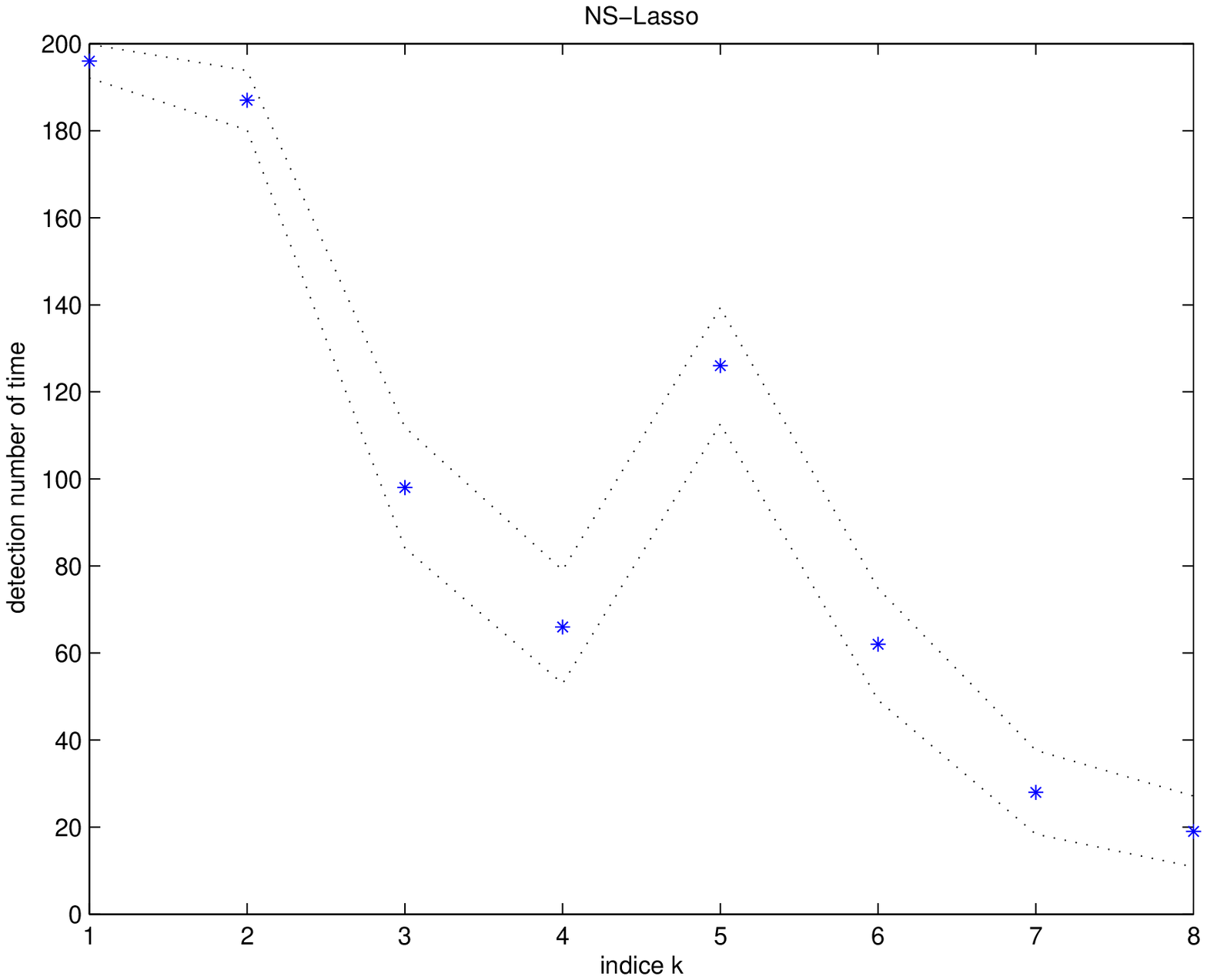}
\includegraphics[height=1.3in,width=1.3in] {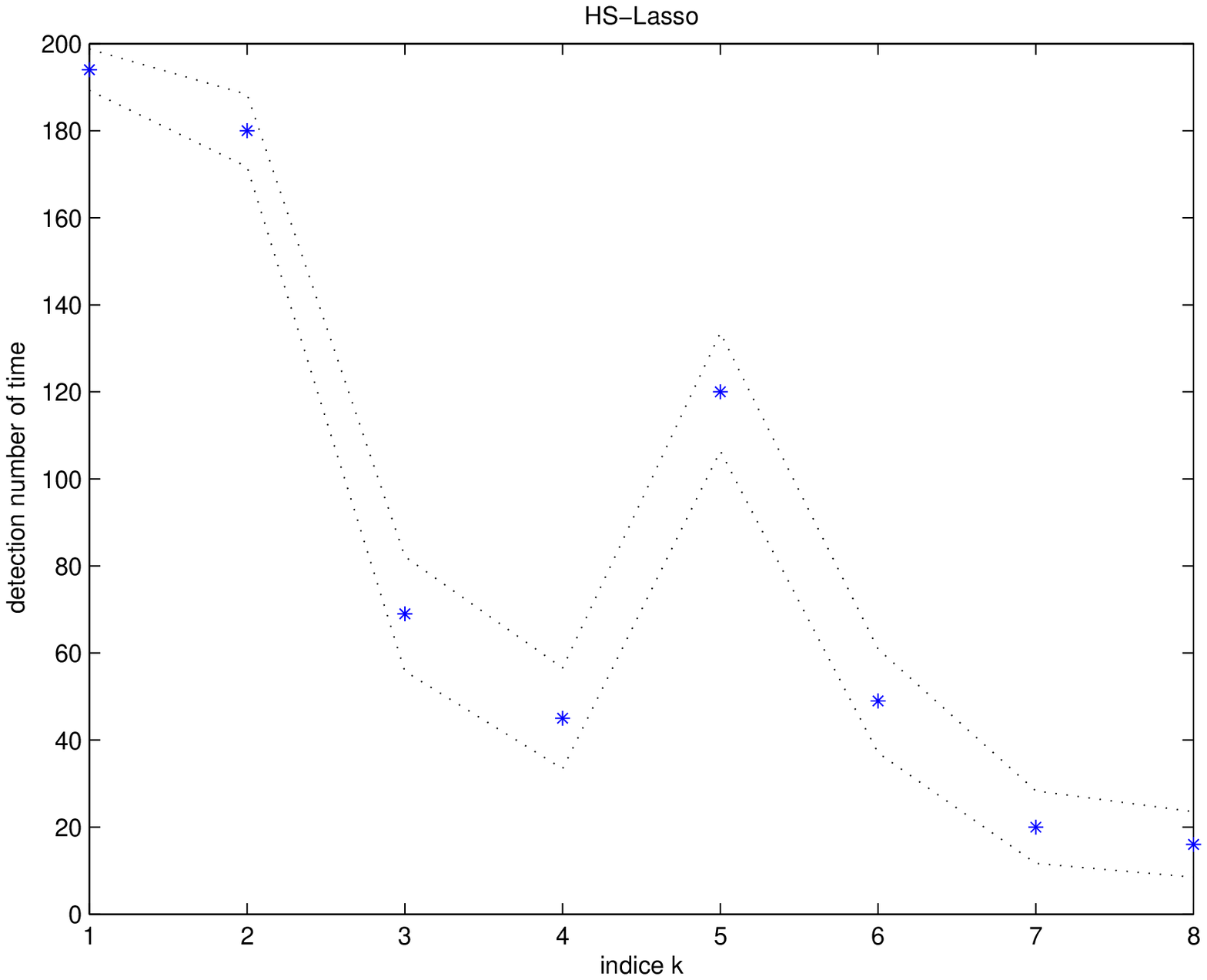}
\end{center}
\end{minipage}
\hspace{.5cm} \hfill\begin{minipage}[t]{0.40\textwidth}
\begin{center}
\includegraphics[height=2.7in,width=2.7in] {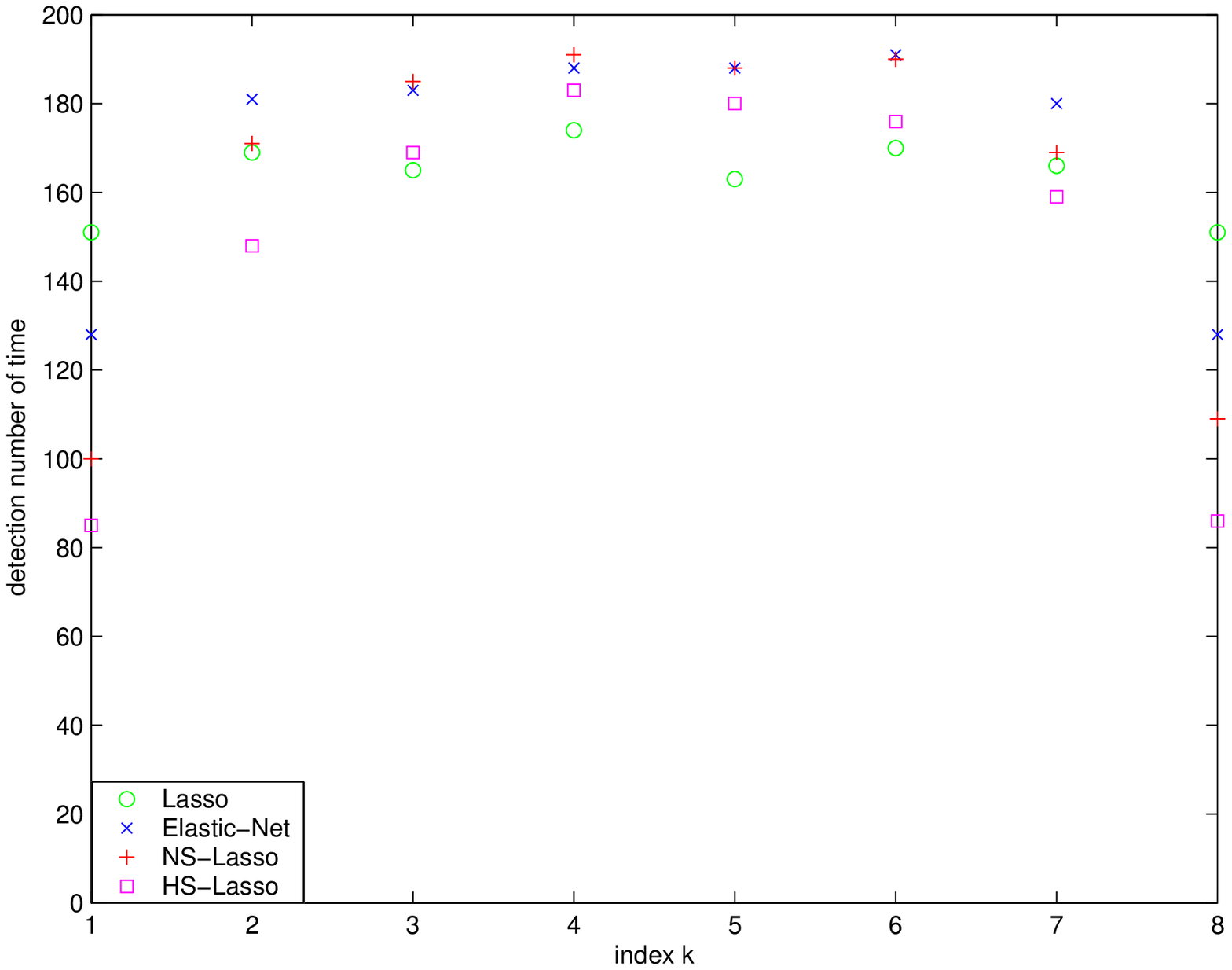}
\end{center}
\end{minipage}

\hfill \vspace{.3cm} \hfill

\begin{minipage}[t]{0.20\textwidth}
\begin{center}
\includegraphics[height=1.3in,width=1.3in] {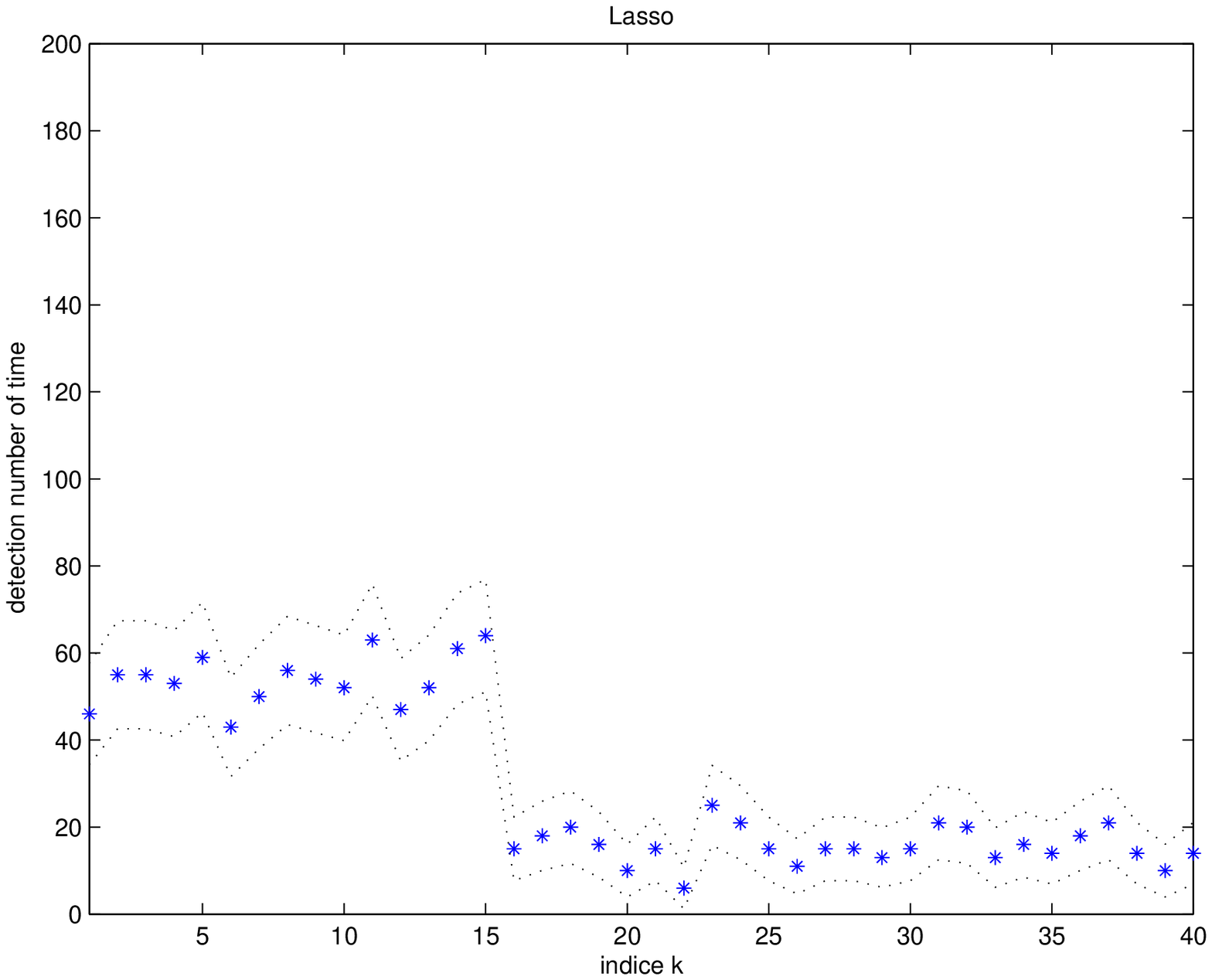}
\includegraphics[height=1.3in,width=1.3in] {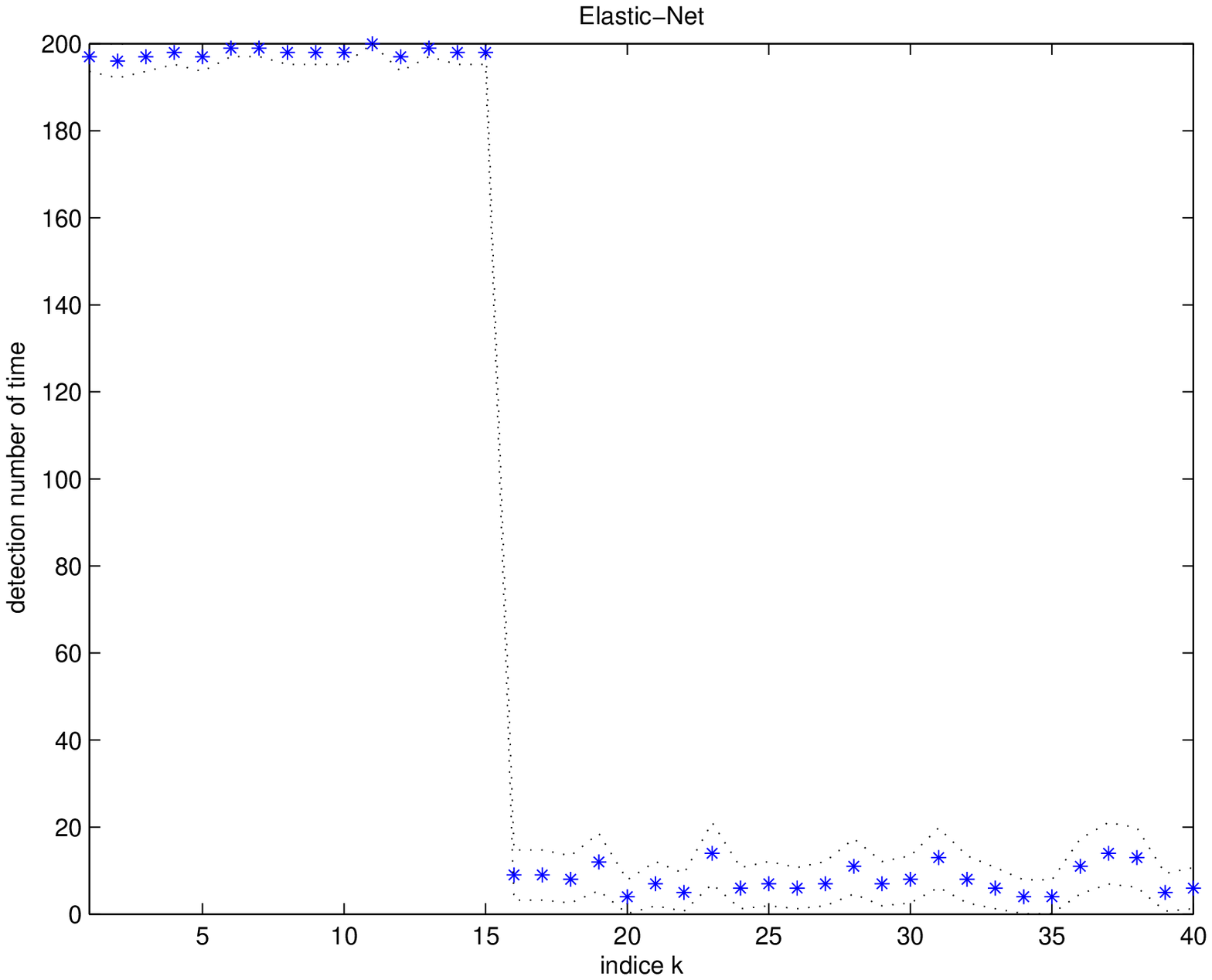}
\end{center}
\end{minipage}
\hfill\begin{minipage}[t]{0.20\textwidth} \begin{center}
\includegraphics[height=1.3in,width=1.3in] {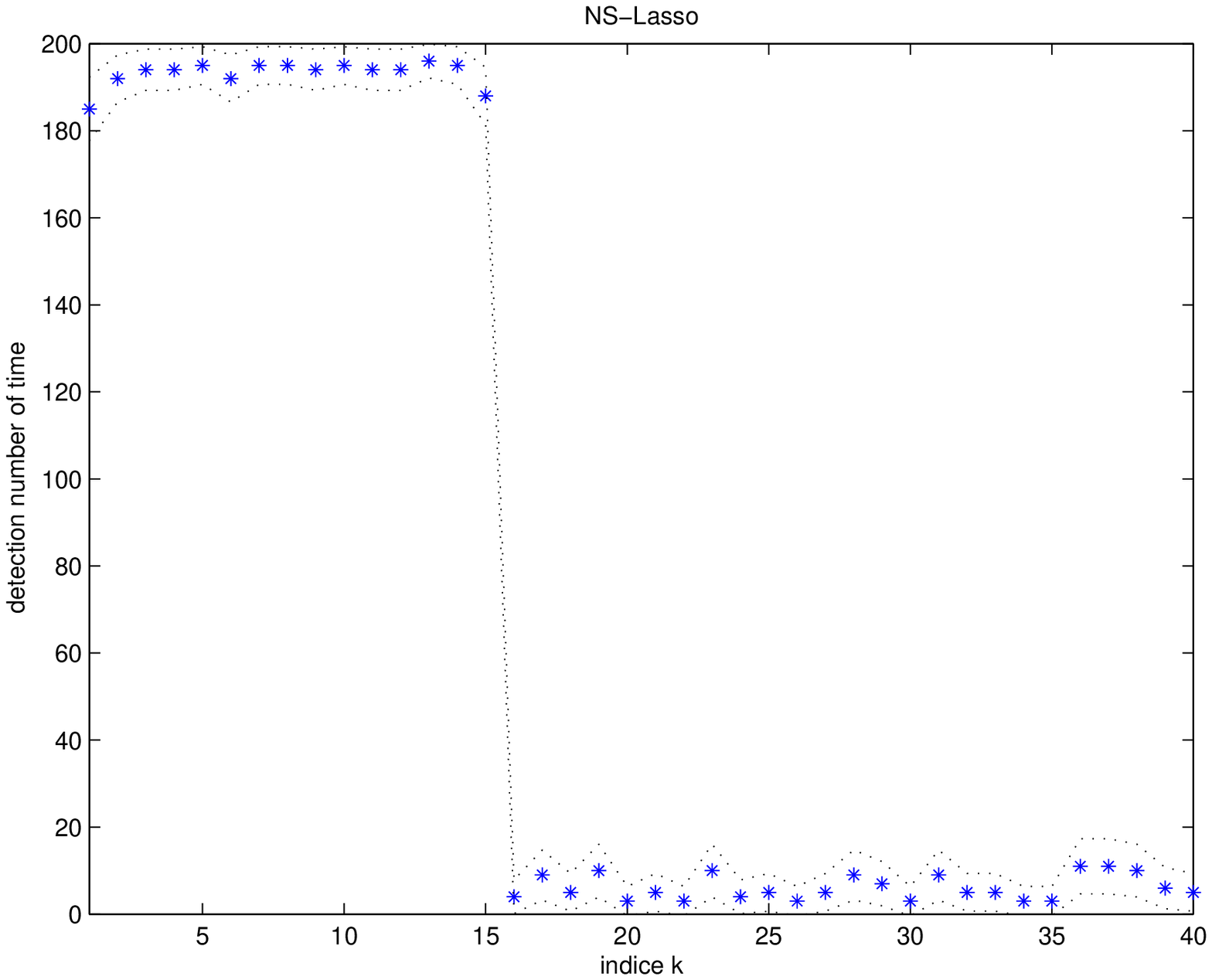}
\includegraphics[height=1.3in,width=1.3in] {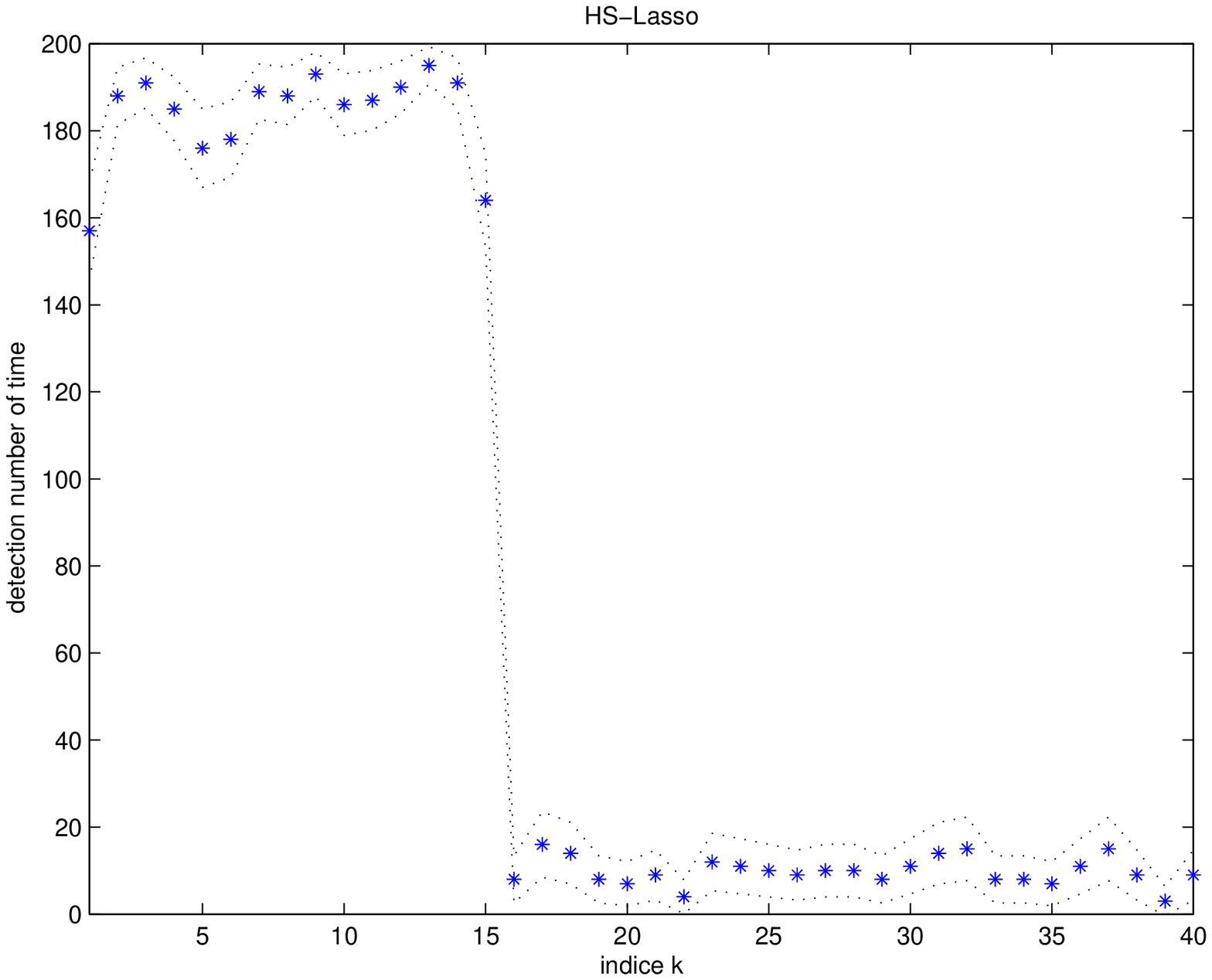}
\end{center}
\end{minipage}
\hspace{.5cm} \hfill\begin{minipage}[t]{0.40\textwidth}
\begin{center}
\includegraphics[height=2.7in,width=2.7in] {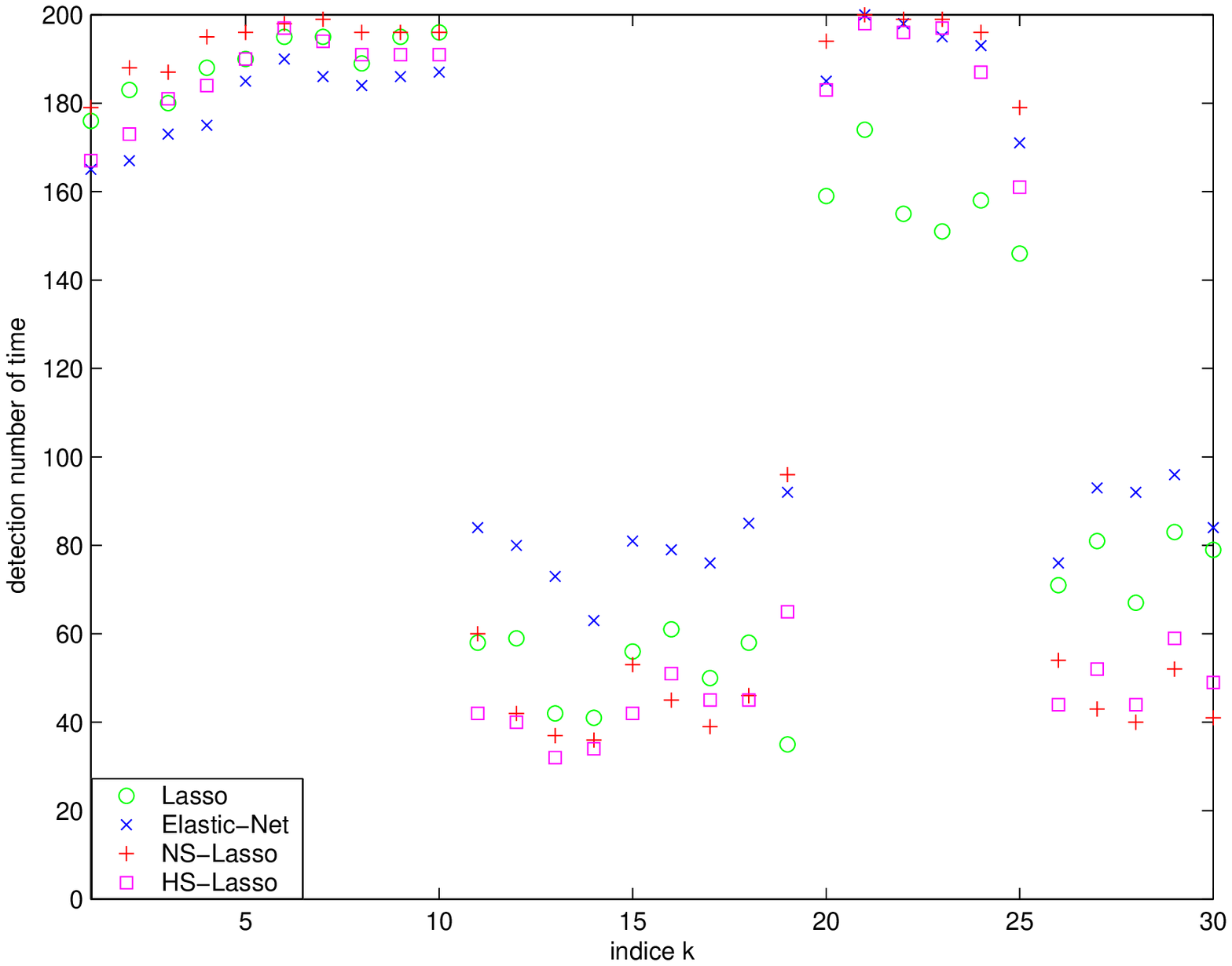}
\end{center}
\end{minipage}
\caption{Number of covariates detections for each procedure in all the examples
(Top-Left: Example (a); Top-Right: Example (b); Bottom-Left: Example (a);
Bottom-Right: Example (b))} \label{fig:PlotA}
\end{figure}

As the Lasso is a special case of the S-Lasso and the Elastic-Net, the Lasso
$\mathop{\rm BIC}$~error (Figure~\ref{fig:BIC}) is always larger than the
$\mathop{\rm BIC}$~error for the other methods. These two seem to have
equivalent $\mathop{\rm BIC}$~errors. When considering the test~error
(Figure~\ref{fig:Test}), it seems again that all the procedures are similar in
all of the examples. They manage to produce good prediction independently of the
sparsity of the model.

The more attractive aspect concerns variable selection. For this purpose we
treat each example separately. \\
Example (a): the Elastic-Net selects a model which is too large
(Table~\ref{NoN-Zero}). This is reflected by the worst $\mathop{\rm SNR}$ (Table~\ref{ratio}). As a consequence, we
can observe in Figure~\ref{fig:PlotA} that it also includes the second
covariate more often than the other procedures. This is due to the "grouping
effect" as the first covariate is relevant. For similar reasons, the S-Lasso
often selects the second covariate. However, this covariate is less selected than
by the Elastic-Net as the S-Lasso seems to be a little bit disturbed by the
third covariate which is irrelevant. This aspect of the S-Lasso procedure is
also present in the selection of the covariate $5$ as its neighbor covariates $4$
and $6$ are irrelevant. We can also observe that the S-Lasso procedure is the
one which selects less often irrelevant covariates when these covariates are far away from relevant ones (in term of indices distance). Finally, even if the Lasso
procedure selects less often the relevant covariates than the Elastic-Net and
the S-Lasso procedures, it also has as good $\mathop{\rm SNR}$. The Lasso presents good
selection performances in this example.\\
Example (b): we can see in Figure~\ref{fig:PlotA} how the S-Lasso and
Elastic-Net selection depends on how the covariates are ranked. They both select
more covariates in the middle (that is covariates $2$ to $7$) than the ones in
the borders (covariates $1$ and $8$) than the Lasso. We also remark that this
aspect is more emphasized for the S-Lasso than
for the Elastic-Net.\\
Example (c): the Lasso procedure performs poorly. It selects more noise
covariates and less relevant ones than the other procedures
(Figure~\ref{fig:PlotA}). It also has the worst
$\mathop{\rm SNR}$ (Table~\ref{ratio}). In this example,
Figure~\ref{fig:PlotA} also shows that the Elastic-Net selects more often
relevant covariates than the S-Lasso procedures but it also selects more noise
covariates than the NS-lasso procedure. Then even if the Elastic-Net has very
good performance in variable selection, the NS-Lasso procedure has similar
performances with a close $\mathop{\rm SNR}$ (Table~\ref{ratio}). The NS-Lasso appears to have very good performance in this
example. However, it selects again less
often relevant covariates at the border than the Elastic-Net.\\
Example (d): we decompose the study into two parts. First, the independent part
which considers covariates $\xi_{1},\ldots,\xi_{10}$ and $\xi_{26},\ldots,\xi_{30}$. The
second part considers the other covariates which are dependent. Regarding the
independent covariates, Figure~\ref{fig:PlotA} shows that all the procedures
perform roughly in the same way, though the S-Lasso procedure enjoys a slightly
better selection (in both relevant and noise group of covariates). For the
dependent and relevant covariates, the Lasso performs worst than the other
procedures. It selects clearly less often these relevant covariates. As in
example (c), the reason is that the Lasso modification of the LARS algorithm
tends to select only one representative of a group of highly correlated
covariates. The high value of the $\mathop{\rm SNR}$ for
the Lasso (when compared to the Elastic-Net) is explained by its good
performance when it treat noise covariates. In this example the Elastic-Net
correctly selects relevant covariates but it is also the procedure which
selects the more noise covariates and has the worst
$\mathop{\rm SNR}$. We also note that both the
NS-Lasso and HS-Lasso outperform the Lasso and Elastic-Net. This gain is
emphasized especially in the center of the groups. Observe that for the
covariates $\xi_{20},\xi_{21},\xi_{25}$ and $\xi_{26}$ (that is the borders), the
NS-Lasso and HS-Lasso have slightly worst performance than in the center of the
groups. This is again due to the attraction we imposed by the fusion penalty
$(\ref{critere_S-lasso})$ in the S-Lasso criterion.\\
\\
\textit{Conclusion of the experiments.} The S-Lasso procedure seems to respond
to our expectations. Indeed, when successive correlations exist, it tends to select
the whole group of these relevant covariates and not only one representing the group as done by the Lasso procedure. It also appears that the S-Lasso procedure has
very good selection properties according to both relevant and noise covariates.
However it has slightly worst performance in the borders than in the centers of
groups of covariates (due to attractions of irrelevant covariates). It almost
always has a better $\mathop{\rm SNR}$ than the
Elastic-Net, so we can take it as a good challenger for this procedure.

\section{Conclusion}
In this paper, we introduced a new procedure called the Smooth-Lasso which takes into account correlation between successive covariates. We established several theoretical results. The main conclusions are that when $p \leq n$, the S-Lasso is consistent in variable selection and asymptotically normal with a rate lower than $\sqrt{n}$. In the high dimensional setting, we provided a condition related to the coherence mutual condition, under which the thresholded version of the Smooth-Lasso is consistent in variable selection. This condition is fulfilled when correlations between successive covariates exist. Moreover, simulation studies showed that normalized versions of the Smooth-Lasso have nice properties of variable selection which are emphasized when high correlations exist between successive covariates. It appears that the Smooth-Lasso almost always outperforms the Lasso and is a good challenger of the Elastic-Net.\\

\section*{Appendix A.}
Since the matrix $\mathcal{C}_n + \mu_n \tilde{J}$ plays a crucial role in the proves, we use to shorten the notation $ K_n = \mathcal{C}_n + \mu_n \tilde{J}$ and when $p\leq n$ we define $K=\mathcal{C} + \mu \tilde{J}$, its limit.\\
In this appendix we prove the results when $p\leq n$. 
\vspace{0.5cm}
\begin{proof} [\it{Proof of Theorem~\ref{thm:AsymptoticNormality}}]
Let $\Psi_{n}$ be
\begin{eqnarray*}
\Psi_{n}(u)=\|Y-X(\beta^{*} + v_{n} u)\|_{n}^{2} & + &
\lambda_{n}\sum_{j=1}^{p} |\beta_{j}^{*} + v_{n} u_{j}| \\
& + & \mu_{n}\sum_{j=2}^{p}\left(\beta_{j}^{*}-\beta_{j-1}^{*}+ v_{n}
(u_{j}-u_{j-1})\right)^{2},
\end{eqnarray*}
for $u=(u_{1},\ldots,u_{p})'\in\mathbb{R}^{p}$ and let
$\hat{u}=\argmin_{u}\Psi_{n}(u)$. Let $\varepsilon =
(\varepsilon_{1},\ldots,\varepsilon_{n})'$, we then have
\begin{eqnarray*}
\Psi_{n}(u)-\Psi_{n}(0) & =: & V_{n}(u) \\
& = & v_{n}^{2}\, u'\left(\frac{X'X}{n}\right)u - 2 \frac{v_{n}}{\sqrt{n}}
\frac{\varepsilon'X}{\sqrt{n}}u + v_{n}\lambda_{n}
\sum_{j=1}^{p}v_{n}^{-1}\left(|\beta_{j}^{*}+ v_{n} u_{j}|-|\beta_{j}^{*}|\right) \\
&& + v_{n} \mu_{n} \sum_{j=2}^{p} v_{n}^{-1}
\left\{\left(\beta_{j}^{*}-\beta_{j-1}^{*}+ v_{n}(u_{j}-u_{j-1})\right)^{2} -
\left(\beta_{j}^{*}-\beta_{j-1}^{*}\right)^{2}\right\} \\
& = & v_{n}^{2} \left[ u'\left(\frac{X'X}{n}\right)u - \frac{2}{v_{n}\sqrt{n}}
\frac{\varepsilon'X}{\sqrt{n}}u + \frac{\lambda_{n}}{v_{n}}
\sum_{j=1}^{p}v_{n}^{-1}\left(|\beta_{j}^{*}+ v_{n}
u_{j}|-|\beta_{j}^{*}|\right) \right. \\
& &  +  \left. \frac{\mu_{n}}{v_{n}} \sum_{j=2}^{p} v_{n}^{-1}
\left\{\left(\beta_{j}^{*}-\beta_{j-1}^{*}+ v_{n}(u_{j}-u_{j-1})\right)^{2} -
\left(\beta_{j}^{*}-\beta_{j-1}^{*}\right)^{2}\right\} \right] \\
& =&  v_{n}^{2} V_{n}(u).
\end{eqnarray*}
Note that $\hat{u}=\argmin_{u}\Psi_{n}(u)=\argmin_{u}V_{n}(u)$, we then have
to consider the limit distribution of $V_{n}(u)$. First, we have
$\frac{X'X}{n}\rightarrow \mathbf{C}$. Moreover, as $1 /
(v_{n}\sqrt{n})\rightarrow \kappa $ and as given $X$, the random variable
$\frac{\varepsilon'X}{\sqrt{n}}\xrightarrow{\mathcal{D}} W$, with
$W\sim~\mathcal{N}(0,\sigma^{2}\mathbf{C})$, the Slutsky theorem implies that
\begin{equation*}
\frac{2}{v_{n}\sqrt{n}} \frac{\varepsilon'X}{\sqrt{n}}u
\xrightarrow{\mathcal{D}} 2 \kappa W' u.
\end{equation*}
Now we treat the last two terms. If $\beta_{j}^{*}\neq0$,
\begin{equation*}
v_{n}^{-1}\left(|\beta_{j}^{*}+ v_{n} u_{j}|-|\beta_{j}^{*}|\right) \rightarrow
u_{j}\mathop{\rm Sgn}(\beta_ {j}^{*}),
\end{equation*}
and is equal to $\abs{u_{j}}$ otherwise. Then, as
\begin{equation*}
\frac{\lambda_{n}}{v_{n}} \sum_{j=1}^{p}v_{n}^{-1}\left(|\beta_{j}^{*}+
v_{n} u_{j}|-|\beta_{j}^{*}|\right) \rightarrow \lambda\sum_{j=1}^{p}
\left\{u_{j}\mathop{\rm Sgn}(\beta_ {j}^{*})\mathbb{I}(\beta_{j}^{*}\neq0) +
\abs{u_{j}}\mathbb{I}(\beta_{j}^{*}=0)\right\},
\end{equation*}
For the remaining term, we show that if $\beta_{j}~\neq~\beta_{j-1}$,
\begin{equation*}
v_{n}^{-1} \left\{\left(\beta_{j}^{*}-\beta_{j-1}^{*}+
v_{n}(u_{j}-u_{j-1})\right)^{2} -
\left(\beta_{j}^{*}-\beta_{j-1}^{*}\right)^{2}\right\} \rightarrow
2(u_{j}-u_{j-1})(\beta_{j}^{*}-\beta_{j-1}^{*}),
\end{equation*}
and is equal to $\frac{(u_{j}-u_{j-1})^{2}}{n}$ otherwise. But
$\mu_{n}$ converge to $0$, implies that
\begin{equation*}
\frac{\mu_{n}}{v_{n}} \sum_{j=2}^{p} v_{n}^{-1}
\left\{\left(\beta_{j}^{*}-\beta_{j-1}^{*}+ v_{n}(u_{j}-u_{j-1})\right)^{2} -
\left(\beta_{j}^{*}-\beta_{j-1}^{*}\right)^{2}\right\} \rightarrow
\end{equation*}
\begin{equation*}
2\mu\sum_{j=2}^{p} \left\{(u_{j}-u_{j-1})
(\beta_{j}^{*}-\beta_{j-1}^{*}) \mathbb{I}(\beta_{j}^{*}
\neq\beta_{j-1}^{*})\right\}.
\end{equation*}
Therefore we have $V_{n}(u)\rightarrow V(u)$ in probability, for every $u\in\mathbb{R}^{p}$.
And since $\mathbf{C}$ is a positive defined matrix, $V(u)$ has a unique
minimizer. Moreover as $V_{n}(u)$ is convex, standard $M$-estimation results
\cite{VanderVaart-Asympt} lead to:
$\hat{u}_{n}\rightarrow\argmin_{u}V(u)$.
\end{proof}
\vspace{0.5cm}
\begin{proof} [\it{Proof of Theorem~\ref{thm:Sufficient-Var-consistant}}]
We begin by giving two results which we will use in our proof. The first one
concerns the optimality conditions of the S-Lasso estimator. Recall that by
definition
\begin{equation*}
\hat{\beta}^{SL}=\argmin_{\beta\in\mathbb{R}^{p}}
\left\|Y-X\beta\right\|_{n}^{2}
 + \lambda_{n} | \beta |_{1} +
\mu_{n}\beta' \widetilde{J}\beta.
\end{equation*}
Note $f(a)|_{a=a_{0}}$ the evaluation of the function $f$ at the point $a_{0}$.
As the above problem is a non-differentiable convex problem, classical tools
lead to the following optimality conditions for the S-Lasso estimator:
\begin{lm}\label{KKT}
The vector
$\hat{\beta}^{SL}=(\hat{\beta}_{1}^{SL},\ldots,\hat{\beta}_{p}^{SL})'$ is the
S-Lasso estimate as defined in
\eqref{eq:penalized-risk}-\eqref{critere_S-lasso} if and only if
\begin{eqnarray}
\label{KKT1}
\left.\frac{\left\|Y-X\beta\right\|_{n}^{2}+\mu_{n}\beta'
\widetilde{J}\beta}{d\beta_{j}}\right|_{\beta_{j}=\hat{\beta}_{j}^{SL}}&=&-\lambda_{n} \mathop{\rm Sgn}(\hat{\beta}_{j}^{SL})  \quad\quad \text{for} \,\, j:\, \hat{\beta}_{j}^{SL} \ne 0,\\
\label{KKT2}
\abs{\left. \frac{\left\|Y-X\beta\right\|_{n}^{2} +
\mu_{n}\beta' \widetilde{J}\beta}{d\beta_{j}}\right|_{\beta_{j} =
\hat{\beta}_{j}^{SL}}}&\leq & \lambda_{n}\quad\quad\quad\quad\quad\quad\quad \text{for}\,\, j:\,
\hat{\beta}_{j}^{SL}= 0.
\end{eqnarray}
\end{lm}

Recall that $\mathcal{A}^{*}=\{j:\beta_{j}^{*} \ne 0 \}$, the second result
states that if we restrict ourselves to the covariates which we are after (i.e.
indexes in $\mathcal{A}^{*}$), we get a consistent estimate as soon as the
regularization parameters $\lambda_{n}$ and $\mu_{n}$ are properly chosen.
\begin{lm}\label{lm:RestrictedConsistent}
Let $\tilde{\beta}_{\mathcal{A}^{*}}$ a minimizer of
\begin{equation*}
\left\|Y-X_{\mathcal{A}^{*}}\beta_{\mathcal{A}^{*}}\right\|_{n}^{2}
 + \lambda_{n} \sum_{j\in \mathcal{A}^{*}} | \beta_{j}| +
\mu_{n}\beta_{\mathcal{A}^{*}}'
\widetilde{J}_{\mathcal{A}^{*},\mathcal{A}^{*}}\beta_{\mathcal{A}^{*}}.
\end{equation*}
If $\lambda_{n} \rightarrow 0$ and $\mu_{n} \rightarrow 0$ , then
$\tilde{\beta}_{\mathcal{A}^{*}}$ converges to $\beta_{\mathcal{A}^{*}}^{*}$ in
probability.
\end{lm}
This lemma can be see as a special and restricted case of
Theorem~\ref{thm:AsymptoticNormality}. We now prove
Theorem~\ref{thm:Sufficient-Var-consistant}. Let
$\tilde{\beta}_{\mathcal{A}^{*}}$ as in Lemma~\ref{lm:RestrictedConsistent}. We
define an estimator $\tilde{\beta}$ by extending
$\tilde{\beta}_{\mathcal{A}^{*}}$ by zeros on $(\mathcal{A}^{*})^{c}$. Hence,
consistency of $\tilde{\beta}$ is ensure as a simple consequence of
Lemma~\ref{lm:RestrictedConsistent}. Now we need to prove that with probability
tending to one, this estimator is optimal for the
problem~\eqref{eq:penalized-risk}-\eqref{critere_S-lasso}. That is the optimal
conditions~\eqref{KKT1}-\eqref{KKT2} are fulfilled with probability tending to
one.

From now on, we denote $\mathcal{A}$ for $\mathcal{A}^{*}$. By definition of
$\tilde{\beta}_{\mathcal{A}}$, the optimality condition~\eqref{KKT1} is
satisfied. We now must check the optimality condition~\eqref{KKT2}. Combining
the fact that $Y=X\beta^{*} + \varepsilon$ and the convergence of the matrix
$X'X/n$ and the vector $\varepsilon'X / \sqrt{n}$, we have
\begin{equation}\label{eq:AsympEquiv}
n^{-1}(X'Y - X'X_{\mathcal{A}} \tilde{\beta}_{\mathcal{A}}) =
\mathbf{C}_{.,\mathcal{A}}(\beta_{\mathcal{A}}^{*} -
\tilde{\beta}_{\mathcal{A}}) + \mathcal{O}_{p}(n^{-1/2}).
\end{equation}
Moreover, the optimality condition~\eqref{KKT1} for the estimator
$\tilde{\beta}$ can be written as
\begin{equation}\label{eq:tildeBeta_KK1}
n^{-1}(X_{.,\mathcal{A}}'Y - X_{.,\mathcal{A}}'X_{.,\mathcal{A}}
\tilde{\beta}_{\mathcal{A}}) = \frac{\lambda_{n}}{2} \mathop{\rm
Sgn}(\tilde{\beta}_{\mathcal{A}}) - \mu_{n}
\widetilde{J}_{\mathcal{A},\mathcal{A}}(\beta_{\mathcal{A}}^{*} -
\tilde{\beta}_{\mathcal{A}}) + \mu_{n}
\widetilde{J}_{\mathcal{A},\mathcal{A}}\beta_{\mathcal{A}}^{*}.
\end{equation}
Combining \eqref{eq:AsympEquiv} and \eqref{eq:tildeBeta_KK1}, we easily obtain
\begin{equation*}
(\beta_{\mathcal{A}}^{*} - \tilde{\beta}_{\mathcal{A}}) =
(\mathbf{C}_{\mathcal{A},\mathcal{A}} + \mu_{n}
\widetilde{J}_{\mathcal{A},\mathcal{A}})^{-1} \left(\frac{\lambda_{n}}{2}
\mathop{\rm Sgn}(\tilde{\beta}_{\mathcal{A}}) + \mu_{n}
\widetilde{J}_{\mathcal{A},\mathcal{A}}\beta_{\mathcal{A}}^{*}\right) +
\mathcal{O}_{p}(n^{-1/2}).
\end{equation*}
Since $\tilde{\beta}$ is consistent and $\lambda_{n} n^{1/2} \rightarrow
\infty$, for each $j\in \mathcal{A}^{c}$, the left hand side in the optimality
condition~\eqref{KKT2}
\begin{equation*}
\frac{1}{\lambda_{n}n}(\xi_{j}'Y - \xi_{j}'X_{.,\mathcal{A}}
\tilde{\beta}_{\mathcal{A}}) - \frac{\mu_{n}}{\lambda_{n}}
\widetilde{J}_{j,\mathcal{A}} \tilde{\beta}_{\mathcal{A}} =: L_{j}^{(n)},
\end{equation*}
converges in probability to
\begin{equation*}
\mathbf{C}_{j,\mathcal{A}} (K_{\mathcal{A},\mathcal{A}})^{-1} \left(2^{-1} \mathop{\rm
Sgn}(\beta_{\mathcal{A}}^{*}) + \frac{\mu}{\lambda}
\widetilde{J}_{\mathcal{A},\mathcal{A}}\beta_{\mathcal{A}}^{*}\right) -
\frac{\mu}{\lambda} \widetilde{J}_{j,\mathcal{A}}
\beta_{\mathcal{A}}^{*} =: L_{j}.
\end{equation*}
By condition~\eqref{SufficientCondition}, this quantity is strictly smaller
than one. Then
\begin{equation*}
\lim_{n\rightarrow \infty} \mathbb{P}\left( \forall j\in \mathcal{A}^{c},\, |
L_{j}^{(n)} | \leq 1  \right) \geq \prod_{j\in \mathcal{A}^{c}}
\mathbb{P}\left( | L_{j} | \leq 1  \right) = 1,
\end{equation*}
which ends the proof.
\end{proof}
\vspace{0.5cm}
\begin{proof} [\it{Proof of Theorem~\ref{thm:Necessary-Var-consistant}}]
We prove the theorem by contradiction by assuming that there exists a
$j\in(\mathcal{A}^{*})^{c}$ such that there exists a $i \in \mathcal{A}^{*}$
and
\begin{equation*}
|\Omega_{j}(\lambda,\mu,\mathcal{A}^{*},\beta^{*})| >1,
\end{equation*}
where the $\Omega_{j}$ are given by \eqref{eq:OmegaCondi}. Since $\mathcal{A}_{n} =
\mathcal{A}^{*}$ with probability tending to one, optimality
condition~\eqref{KKT1} implies
\begin{equation}\label{eq:Explicit_beta}
\hat{\beta}_{\mathcal{A}}^{SL} =
((K_n)_{\mathcal{A},\mathcal{A}})^{-1} \left(\frac{X_{.,\mathcal{A}}'Y}{n} - \frac{\lambda_{n}}{2} \mathop{\rm
Sgn}(\hat{\beta}_{\mathcal{A}}^{SL})\right).
\end{equation}
Using this expression of $\hat{\beta}_{\mathcal{A}}^{SL}$ and $Y =
X_{.,\mathcal{A}}\beta_{\mathcal{A}}^{*} + \varepsilon$, then for every
$j\in\mathcal{A}^{c}$,
\begin{eqnarray*}
\frac{\xi_{j}'Y}{n} -
\frac{\xi_{j}'X_{.,\mathcal{A}}\hat{\beta}_{\mathcal{A}}^{SL}}{n} & = &
\frac{\xi_{j}'Y}{n} - \frac{\xi_{j}'X_{.,\mathcal{A}}}{n}
((K_n)_{\mathcal{A},\mathcal{A}})^{-1} \frac{X_{.,\mathcal{A}}'Y}{n} \\ & & + \frac{\lambda_{n}}{2}
\frac{\xi_{j}'X_{.,\mathcal{A}}}{n}
((K_n)_{\mathcal{A},\mathcal{A}})^{-1}  \mathop{\rm
Sgn}(\hat{\beta}_{\mathcal{A}}^{SL}) \\ 
&=& \frac{\xi_{j}'Y}{n} - \frac{\xi_{j}'X_{.,\mathcal{A}}}{n}
((K_n)_{\mathcal{A},\mathcal{A}})^{-1} \frac{X_{.,\mathcal{A}}'\varepsilon}{n} - \frac{\xi_{j}'X_{.,\mathcal{A}}}{n}
\beta_{\mathcal{A}}^{*}  \\ & & + \frac{\xi_{j}'X_{.,\mathcal{A}}}{n}
((K_n)_{\mathcal{A},\mathcal{A}})^{-1} \left(\frac{\lambda_{n}}{2}\mathop{\rm
Sgn}(\hat{\beta}_{\mathcal{A}}^{SL}) + \mu_{n}
\widetilde{J}_{\mathcal{A},\mathcal{A}} \beta_{\mathcal{A}}^{*} \right).
\end{eqnarray*}
Therefore,
\begin{equation*}
n^{-1}(\xi_{j}'Y - \xi_{j}'X_{.,\mathcal{A}}\hat{\beta}_{\mathcal{A}}^{SL}) -
\mu_{n} \widetilde{J}_{j,\mathcal{A}}\beta_{\mathcal{A}}^{SL}  =
A_{n} + B_{n},
\end{equation*}
with
\begin{equation*}
\left\{
\begin{array}{l}\label{Z}
A_{n}=\frac{\xi_{j}'Y}{n} - \frac{\xi_{j}'X_{.,\mathcal{A}}}{n}
((K_n)_{\mathcal{A},\mathcal{A}})^{-1} \frac{X_{.,\mathcal{A}}'\varepsilon}{n} - \frac{\xi_{j}'X_{.,\mathcal{A}}}{n}
\beta_{\mathcal{A}}^{*} \\
B_{n} = \frac{\xi_{j}'X_{.,\mathcal{A}}}{n}
((K_n)_{\mathcal{A},\mathcal{A}})^{-1} \left(\frac{\lambda_{n}}{2}\mathop{\rm
Sgn}(\hat{\beta}_{\mathcal{A}}^{SL}) + \mu_{n}
\widetilde{J}_{\mathcal{A},\mathcal{A}} \beta_{\mathcal{A}}^{*} \right) -
\mu_{n} \widetilde{J}_{j,\mathcal{A}}\hat{\beta}_{\mathcal{A}}^{SL}.
\end{array}
\right.
\end{equation*}
We treat this two terms separately. First as $\hat{\beta}_{\mathcal{A}}^{SL}$
converges in probability to $\beta_{\mathcal{A}}^{*}$ and empirical covariance
matrices convergence, the sequence $B_{n}/\lambda_{n}$ converges to
\begin{equation*}
B = \mathbf{C}_{j,\mathcal{A}} (K_{\mathcal{A},\mathcal{A}})^{-1} 
(2^{-1}\lambda \mathop{\rm Sgn}(\beta_{\mathcal{A}}^{*}) + \mu
\lambda^{-1} \widetilde{J}_{\mathcal{A},\mathcal{A}}
\beta_{\mathcal{A}}^{*} ) - \mu \lambda^{-1}
\widetilde{J}_{j,\mathcal{A}}\beta_{\mathcal{A}}^{*}.
\end{equation*}
By assumption $|B|>1$. This implies that
$\mathbb{P}\left(B_{n}/\lambda_{n} \geq (1+|B|)/2 \right)$ converges to
one.

With regard to the other term, since $Y=X\beta^{*}+ \varepsilon $ we have
\begin{eqnarray*}
A_{n} & = & \frac{\xi_{j}'\varepsilon}{n} - \frac{\xi_{j}'X_{.,\mathcal{A}}}{n}
((K_n)_{\mathcal{A},\mathcal{A}})^{-1} \frac{X_{.,\mathcal{A}}'\varepsilon}{n} \\ & = & n^{-1} \sum_{k=1}^{n}
\varepsilon_{k} (x_{k,j} - \mathbf{C}_{j,\mathcal{A}}
(K_{\mathcal{A},\mathcal{A}})^{-1}  x_{k,\mathcal{A}}' )
+ o_{p}(n^{-1/2}) \\ & = & n^{-1} \sum_{k=1}^{n} c_n  +  o_{p}(n^{-1/2}) = C_n
+o_{p}(n^{-1/2}),
\end{eqnarray*}
where $c_n$ are i.i.d. random variables with mean $0$ and variance:
\begin{eqnarray*}
s^{2} = \mathbb{V}ar(c_k) & = & \mathbb{E}(c_{k}^{2}) =
\mathbb{E}[\mathbb{E}(c_{k}^{2}|X)] \\ & = & \mathbb{E}\left[
\mathbb{E}(\varepsilon_{k}^{2}|X) (x_{k,j} - \mathbf{C}_{j,\mathcal{A}}
(K_{\mathcal{A},\mathcal{A}})^{-1} x_{k,\mathcal{A}}')^{2} \right] \\ & = & \sigma^{2} \mathbb{E}\left[
\mathbf{C}_{j,j} +
\mathbf{C}_{j,\mathcal{A}} (K_{\mathcal{A},\mathcal{A}})^{-1}
\mathbf{C}_{\mathcal{A},\mathcal{A}} (K_{\mathcal{A},\mathcal{A}})^{-1}
\mathbf{C}_{\mathcal{A},j} \right. \\ & & \left. -2
\mathbf{C}_{j,\mathcal{A}} (K_{\mathcal{A},\mathcal{A}})^{-1}
\mathbf{C}_{\mathcal{A},j}\right].
\end{eqnarray*}
Thus, by the central limit theorem, $n^{1/2}C_n$ is asymptotically normal with
mean $0$ and covariance matrix $s^{2}/n$, which is finite. Thus
$\mathbb{P}(n^{1/2}A_{n} > 0)$ converges to $1/2$.

Finally, $\mathbb{P}((A_{n} + B_{n})/ \lambda_{n} > (1+|B|)/2)$ is
asymptotically bounded below by $1/2$. Thus $|(A_{n} + B_{n})/
\lambda_{n}|$ is asymptotically bigger than $1$ with a positive
probability, that is to say the optimality condition~\eqref{KKT2} is not
satisfied. Then $\hat{\beta}^{SL}$ is not optimal. We get a contradiction,
which concludes the proof.
\end{proof}

\section*{Appendix B.}
In this appendix we mainly prove the results when $p\geq n$.
\begin{proof} [\it{Proof of Theorem~\ref{xasp}}]Using the definition of the penalized estimator~\eqref{eq:penalized-risk}--\eqref{critere_S-lasso}, for any $\beta \in \mathbb{R}^{p}$, we have
\begin{eqnarray*}
\lefteqn{\| X\hat{\beta}^{SL} - X\beta ^*\|_{n}^{2} - \frac{2}{n}\sum_{i=1}^{n}
\varepsilon_{i}x_i \hat{\beta}^{SL} + \lambda_n |\hat{\beta}^{SL}|_{1}
+ \mu_n  (\hat{\beta}^{SL})'\widetilde{J}\hat{\beta}^{SL}} \nonumber \\ & \leq & \|
X\beta -X\beta^{*} \|_{n}^{2} - \frac{2}{n}\sum_{i=1}^{n} \varepsilon_{i} 
x_i\beta + \lambda_n |\beta|_{1}  + \mu_n  \beta'\widetilde{J}\beta.
\end{eqnarray*}
Therefore, if we chose $\beta =\beta^*$, we obtain the following inequalities:
\begin{eqnarray}\label{voit}
\| X\hat{\beta}^{SL}  - X\beta ^*\|_{n}^{2} & \leq &  \lambda_n \sum_{j=1}^{p}
\left(|\beta_{j}^{*}| - |\hat{\beta}_{j}^{SL} |\right)  + \frac{2}{n}\sum_{i=1}^{n}
\varepsilon_{i}x_i(\hat{\beta}^{SL} -\beta^*) \nonumber \\ && + \mu_n 
(\beta^{*'}  \widetilde{J}\beta^{*} - (\hat{\beta}^{SL} )'\widetilde{J}
\hat{\beta}^{SL} )  
\nonumber \\ & \leq &   \lambda_n \sum_{j=1}^{p}
\left(|\beta_{j}^{*}| - |\hat{\beta}_{j}^{SL} |\right)  + \frac{2}{n}\sum_{i=1}^{n}
\varepsilon_{i}x_i(\hat{\beta}^{SL} -\beta^*) \nonumber \\ && + \mu_n 
\beta^{*'}  \widetilde{J}\beta^{*},
\end{eqnarray}
as $\beta ' \widetilde{J}\beta \geq 0$ for any $\beta\in \mathbb{R}^{p}$. In order to control \eqref{voit}, we use in a first time Assumption~(A1) so that $\displaystyle{\mu_n  \beta^{*'}  \widetilde{J}\beta^{*} \leq L_1\kappa_{2} \sigma^{2} \frac{\log (p) |\mathcal{A^*}|}{n}}$. Second we bound the residual term in the same way as in \cite{SparLassBTW07}. Then, we only present here the main lines. Recall that $\mathcal{A} = \mathcal{A}^{*} = \{j:\,\beta_j^*\neq 0\}$. Then, on the event $\Lambda_{n,p}=\left \lbrace \max_{j=1,\ldots,p} 4|V_j| \le \lambda_{n} \right \rbrace$ with $V_{j}=n^{-1}\sum_{i=1}^{n}x_{i,j}\varepsilon_{i}$, we have
\begin{eqnarray}\label{joi2}
\| X\hat{\beta}^{SL} - X\beta ^* \|_{n}^{2}  + 2^{-1} \lambda_{n} \sum_{j=1}^{p} 
\abs{\hat{\beta}_{j}^{SL}-\beta^*_{j}} \leq \lambda_{n} \sum_{j\in \mathcal{A}} 
\abs{\hat{\beta}_{j}^{SL}-\beta^*_{j}} + L_1\kappa_{2} \sigma^{2} \frac{\log (p) |\mathcal{A}|}{n}.
\end{eqnarray}
This inequality is obtained thanks to the fact that $|\beta_{j}^{*} - \hat{\beta}_{j}^{SL} | + |\beta_{j}^{*}| - |\hat{\beta}_{j}^{SL} |=0 $ for any $j\notin \mathcal{A}$ and to the triangular inequality. The rest of the proof consists in bounding this term $\lambda_{n} \sum_{j\in \mathcal{A}} 
\abs{\hat{\beta}_{j}^{SL}-\beta^*_{j}}$. Using similar arguments as in \cite{SparLassBTW07}, we can write
\begin{eqnarray}\label{eq:jio}
\sum_{j\in \mathcal{A}}  (\hat\beta_j^{SL}-\beta^*_j)^2 &\leq &  \| X\hat{\beta}^{SL}-X\beta^*\|_n^2  + 2 \rho_1 \sum_{k\in \mathcal{A}} |\hat
\beta_{k}^{SL}-\beta^*_{k}| \sum_{j=1}^{p} |\hat
\beta_j^{SL}-\beta^*_j|  \nonumber \\
&& -  \rho_1  \left( \sum_{j\in \mathcal{A}} 
|\hat \beta_j^{SL}-\beta^*_j| \right)^{2}.
\end{eqnarray}
But $\left(\sum_{j\in \mathcal{A}} \abs{\hat{\beta}_{j}^{SL}-\beta^*_{j}}\right)^2 \leq |\mathcal{A}| \sum_{j\in \mathcal{A}} (\hat{\beta}_{j}^{SL}-\beta^*_{j})^2$, then
\begin{eqnarray}
\lefteqn{\left(\sum_{j\in \mathcal{A}} 
\abs{\hat{\beta}_{j}^{SL}-\beta^*_{j}}\right)^{2}} \nonumber
\\& \leq    |\mathcal{A}| & \left\{  \| X\hat
\beta^{SL}-X\beta^*\|_n^2 + 2 \, \rho_1 \sum_{k\in  \mathcal{A}}
|\hat \beta_{k}^{SL}-\beta^*_{k}| \sum_{j=1}^{p} |\hat
\beta_j^{SL}-\beta^*_j|  \right. \nonumber \\
&& - \left. \rho_1 \left(  \sum_{j\in\mathcal{A}} |\hat
\beta_j^{SL}-\beta^*_j| \right)^{2} \right\}.
\end{eqnarray}
A simple optimization implies
\begin{equation}\label{eq:bornOptim}
\sum_{j\in \mathcal{A}}  \abs{\hat{\beta}_{j}^{SL}-\beta^*_{j}} \leq
\frac{2 \rho_1 |\mathcal{A}|  \sum_{j=1}^{p} |\hat{
\beta}_j^{SL}-\beta_{j}^{*}|}{1+\rho_1 |\mathcal{A}| } +  \frac{
\sqrt{|\mathcal{A}| } \, \| X\hat \beta^{SL}-X\beta^*\|_n^2}{1+\rho_1 |\mathcal{A}| }.
\end{equation}
Now, use Assumption~(A2) to bound the left hand side of the inequality~\eqref{eq:bornOptim} and combine this to \eqref{joi2} to get
\begin{eqnarray}\label{eq:must}
\| X\hat{\beta}^{SL} - X\beta ^* \|_{n}^{2}  + \lambda_n \sum_{j=1}^{p} 
\abs{\hat{\beta}_{j}^{SL}-\beta^*_{j}}  \leq  
16 \lambda_{n}^{2} |\mathcal{A}| + L_1\kappa_{2} \sigma^{2} \frac{\log (p) |\mathcal{A}|}{n}.
\end{eqnarray}
This proves \eqref{bhu}. Finally \eqref{bhuu} follows directly by dividing by $\lambda_n$ both sides of this last inequality.
A concentration inequality to bound $\mathbb{P}\left(\max_{j=1,\ldots,p} 4|V_j| \le \lambda_n\right)$ allows us to conclude the proof.
\end{proof}
\begin{lm}
\label{probalm}
Let $\Lambda_{n,p}$ be the random event defined by $\Lambda_{n,p}=\left \lbrace \max_{j=1,\ldots,p} 4|V_j| \le \lambda_n \right
\rbrace$ where $V_{j}= n^{-1}\sum_{i=1}^{n}x_{i,j}\varepsilon_{i}$. Let us choose a $\kappa_1>2\sqrt{2}$ and $\lambda_n=\kappa_1\sigma\sqrt{n^{-1}\log(p)}$.
Then
$$\mathbb{P}\left(\max_{j=1,\ldots,p} 4|V_j| \le \lambda_n \right)\geq 1-p^{1-\frac{\kappa_1^{2}}{8}}.$$
\end{lm}
\begin{proof}
Since $V_{j}\sim
\mathcal{N}(0,n^{-1}\sigma^2)$, an elementary Gaussian inequality gives
\begin{eqnarray*}
\mathbb{P}\left( \max_{j=1,\ldots,p} \lambda_n^{-1}|V_j| \ge 4^{-1}\right)
& \le  & p \max_{j=1,\ldots,p} \mathbb{P}\left( \lambda_n^{-1}|V_j| \ge 4^{-1}\right)\\
& \le
& p \exp\left(- \kappa_1^2 \log (p) / 8 \right)\\
&= & p^{1-\kappa_1^2/8}.
\end{eqnarray*}
This ends the proof.
\end{proof}
\begin{proof} [\it{Proof of Theorem~\ref{th:supNorm}}]
Through this proof, for any $a\in \mathbb{R}^{p}$, let us denote by $a_{\mathcal{A}}$, the $p$-dimensional vector such that $(a_{\mathcal{A}})_j = a_j$ if $j\in \mathcal{A}$ and zero otherwise. Moreover, we recall that $K_n  = \mathbf{C}_n + \mu_n \widetilde{J}$. Now, note that we can write the KKT conditions~\eqref{KKT1}-~\eqref{KKT2} as \begin{equation}\label{eq:KKTq2}
\|K_n(\hat{\beta}^{SL} - \beta^{*}) -\frac{X'\varepsilon}{n} + \mu_n \widetilde{J} \beta^* \|_{\infty} \leq \frac{\lambda_n}{2}.
\end{equation}
Recall that $\Lambda_{n,p}=\left \lbrace \max_{j=1,\ldots,p} 2|V_j| \le \lambda_n \right \rbrace$ with $V_j = \frac{\xi_j'\varepsilon}{n}$, then applying \eqref{eq:KKTq2} and Assumption~(A4), we have on $\Lambda_{n,p}$ and for any $j\in\{1,\ldots,p\}$
\begin{align*}
|(K_n)_{j,j} ( \hat{\beta}_{j}^{SL} - \beta_{j}^{*})| & = |\{K_n ( \hat{\beta}^{SL} - \beta^{*})\}_j - \sum_{ \underset{k \not = j}{k=1}}^{p} (K_n)_{j,k} ( \hat{\beta}_k^{SL} - \beta_{k}^{*}) + \mu_n (\widetilde{J}\beta^*)_{j}| \\
& \leq \frac{\lambda_n}{2} + |\frac{\xi_j '\varepsilon}{n}| + \sum_{ \underset{k \not = j}{k=1}}^{p}| (K_n)_{j,k} ( \hat{\beta}_k^{SL} - \beta_{k}^{*})  + \mu_n (\widetilde{J}\beta^*)_{j}| \\
& \leq  \frac{3\lambda_n}{4 } + \frac{1}{3\alpha |\mathcal{A}|} |\hat{\beta}^{SL} - \beta^{*}|_{1}  + \mu_n |(\widetilde{J}\beta^*)_{j}|.
\end{align*}
Then 
\begin{equation}\label{eq:chuit}
\|K_n(\hat{\beta}^{SL}- \beta^{*})\|_{\infty} \leq \frac{3\lambda_n}{4} + \frac{1}{3\alpha |\mathcal{A}|} |\hat{\beta}^{SL} - \beta^{*}|_{1} + \mu_n \|\widetilde{J}\beta^*\|_{\infty} .
\end{equation}
Let us now bound $|\hat{\beta}^{SL} - \beta^{*}|_{1}$. Thanks to \eqref{voit}, we can write
\begin{align*}
& \lambda_n | \hat{\beta}^{SL}|_1 \leq \lambda_n |\beta^*|_1 + \frac{2}{n} \sum_{i=1}^{p} \varepsilon_i x_i (\hat{\beta}^{SL} - \beta^{*}) + \mu_n \beta^{*'}\widetilde{J}\beta^{*} \\ 
\overset{\text{on }\Lambda_{n,p}}{\iff} & \lambda_n | \hat{\beta}^{SL}|_1 \leq  \lambda_n |\beta^*|_1 +  \frac{\lambda_n}{2}|\hat{\beta}^{SL} - \beta^{*}|_1  + \mu_n \beta^{*'}\widetilde{J}\beta^{*}.
\end{align*}
Dividing by $\lambda_n$, and adding $2^{-1}|\hat{\beta}^{SL} - \beta^{*}|_1 -  | \hat{\beta}^{SL}|_1$, we get on the event $\Lambda_{n,p}$
\begin{eqnarray}
& 2^{-1}|\hat{\beta}^{SL} - \beta^{*}|_1  & \leq (|\hat{\beta}^{SL} - \beta^{*}|_1 +  |\beta^*|_1  - | \hat{\beta}^{SL}|_1) + \frac{\mu_n}{\lambda_n} \beta^{*'}\widetilde{J}\beta^{*} \nonumber \\
\iff & |\hat{\beta}^{SL} - \beta^{*}|_1 &  \leq 2 |\hat{\beta}_{\mathcal{A}}^{SL} - \beta_{\mathcal{A}}^{*}|_1  + 2 \frac{\mu_n}{\lambda_n} \beta^{*'}_{\mathcal{A}}\widetilde{J}\beta_{\mathcal{A}^*}^{*} \label{eq:pitsch2}  \\
\iff & |\hat{\beta}^{SL} - \beta^{*}|_1 &  \leq 2  \sqrt{ |\mathcal{A}|}\, \|\hat{\beta}_{\mathcal{A}}^{SL} - \beta_{\mathcal{A}}^{*}\|_2 + 2 \frac{\mu_n}{\lambda_n} \beta^{*'}_{\mathcal{A}}\widetilde{J}\beta_{\mathcal{A}^*}^{*},  \label{eq:pitsch}
\end{eqnarray}
where we used the Cauchy Schwarz inequality in the last line. Combine \eqref{eq:chuit} and \eqref{eq:pitsch}, we easily get
\begin{eqnarray}\label{eq:tisch}
\| \hat{\beta}^{SL} - \beta^{*} \|_{\infty}  \leq  \frac{1}{1 + \mu_n} & \left(  \frac{3\lambda_n}{4} + \frac{2 }{3\alpha |\mathcal{A}|}\sqrt{|\mathcal{A}|}\, \|\hat{\beta}_{\mathcal{A}}^{SL} - \beta_{\mathcal{A}}^{*}\|_2 \right.  \nonumber \\ 
& \left. \quad  + \mu_n \|\widetilde{J}\beta^*\|_{\infty}  + \frac{2 \mu_n}{3\alpha \lambda_n |\mathcal{A}|}  \beta^{*'}_{\mathcal{A}}\widetilde{J}\beta_{\mathcal{A}}^{*}  \right) .
\end{eqnarray}
The final step consists in bounding $\|\hat{\beta}_{\mathcal{A}}^{SL} - \beta_{\mathcal{A}}^{*}\|_2$. First, using the KKT condition~\eqref{eq:KKTq2}, we remark that
$\|K_n(\hat{\beta}^{SL}- \beta^{*}) \|_{\infty} \leq 3\lambda_n / 4 + \mu_n \| \widetilde{J}\beta^*\|_{\infty}$ on $\Lambda_{n,p}$ . This and equation~\eqref{eq:pitsch} lied to
\begin{align}\label{eq:toiu}
(\hat{\beta}^{SL}- \beta^{*})' K_n (\hat{\beta}^{SL}- \beta^{*}) & \leq \| K_n (\hat{\beta}^{SL}- \beta^{*})\|_{\infty} \, | \hat{\beta}^{SL} - \beta^{*} |_{1} \nonumber \\
& \leq (\frac{3\lambda_n}{4} + \mu_n \| \widetilde{J}\beta^*\|_{\infty}) (2  \sqrt{ |\mathcal{A}|} \|\hat{\beta}_{\mathcal{A}}^{SL} - \beta_{\mathcal{A}}^{*}\|_2 + 2 \frac{\mu_n}{\lambda_n} \beta^{*'}_{\mathcal{A}}\widetilde{J}\beta_{\mathcal{A}}^{*}).
\end{align}
On the other hand, using Assumption~(A4), and similar arguments as in~\cite{KarimNormSup},
\begin{eqnarray*}
\frac{(\hat{\beta}_{\mathcal{A}}^{SL} - \beta_{\mathcal{A}}^{*})' K_n (\hat{\beta}_{\mathcal{A}}^{SL} - \beta_{\mathcal{A}}^{*})}{\|\hat{\beta}_{\mathcal{A}}^{SL} - \beta_{\mathcal{A}}^{*}\|_{2}^{2}} & = & \frac{(\hat{\beta}_{\mathcal{A}}^{SL} - \beta_{\mathcal{A}}^{*})'\diag(K_n)(\hat{\beta}_{\mathcal{A}}^{SL} - \beta_{\mathcal{A}}^{*})}{\|\hat{\beta}_{\mathcal{A}}^{SL} - \beta_{\mathcal{A}}^{*}\|_{2}^{2}} \\  &&+\frac{(\hat{\beta}_{\mathcal{A}}^{SL} - \beta_{\mathcal{A}}^{*})'(K_n- \diag (K_n)) (\hat{\beta}_{\mathcal{A}}^{SL} - \beta_{\mathcal{A}}^{*})}{\|\hat{\beta}_{\mathcal{A}}^{SL} - \beta_{\mathcal{A}}^{*}\|_{2}^{2}}  \\ &
\geq & 1 - \frac{1}{3\alpha | \mathcal{A} |}\sum_{j,k=1}^{p} \frac{|(\hat{\beta}_{\mathcal{A}}^{SL} - \beta_{\mathcal{A}}^{*})_j| \, |(\hat{\beta}_{\mathcal{A}}^{SL} - \beta_{\mathcal{A}}^{*})_k|}{\|\hat{\beta}_{\mathcal{A}} - \beta_{\mathcal{A}}^{*}\|_{2}^{2}} \\ &
\geq & 1 -  \frac{1}{3\alpha |\mathcal{A}|} \frac{\|\hat{\beta}_{\mathcal{A}}^{SL} - \beta_{\mathcal{A}}^{*}\|_{1}^{2}}{\|\hat{\beta}_{\mathcal{A}}^{SL} - \beta_{\mathcal{A}}^{*}\|_{2}^{2}},
\end{eqnarray*}
where we used in the second inequality the fact that $\diag(K_n)$ has larger diagonal elements than $1$ since the diagonal elements in $\mathbf{C}_n$ and $\widetilde{J}$ are respectively equal to $1$ and larger than $0$. Now, twice using Assumption~(A4), one deduces
\begin{eqnarray*}
\frac{(\hat{\beta}^{SL} - \beta^{*})' K_n (\hat{\beta}^{SL} - \beta^{*})}{\|\hat{\beta}_{\mathcal{A}}^{SL} - \beta_{\mathcal{A}}^{*}\|_{2}^{2}} & \geq &
\frac{(\hat{\beta}_{\mathcal{A}}^{SL} - \beta_{\mathcal{A}}^{*})'K_n (\hat{\beta}_{\mathcal{A}}^{SL} - \beta_{\mathcal{A}}^{*})}{\|\hat{\beta}_{\mathcal{A}}^{SL} - \beta_{\mathcal{A}}^{*}\|_{2}^{2}} + \frac{(\hat{\beta}_{\mathcal{A}^c}^{SL} - \beta_{\mathcal{A}^c}^{*})' K_n (\hat{\beta}_{\mathcal{A}^c}^{SL} - \beta_{\mathcal{A}^c}^{*})}{\|\hat{\beta}_{\mathcal{A}}^{SL} - \beta_{\mathcal{A}}^{*}\|_{2}^{2}} \\
& \geq & 1 -\frac{1}{3\alpha |\mathcal{A}|} \frac{|\hat{\beta}_{\mathcal{A}}^{SL} - \beta_{\mathcal{A}}^{*}|_{1}^{2}}{\|\hat{\beta}_{\mathcal{A}}^{SL} - \beta_{\mathcal{A}}^{*}\|_{2}^{2}} -  \frac{|\hat{\beta}_{\mathcal{A}}^{SL} - \beta_{\mathcal{A}}^{*}|_{1} |\hat{\beta}_{\mathcal{A}^c}^{SL} - \beta_{\mathcal{A}^c}^{*}|_{1}}{3\alpha |\mathcal{A}| \, \|\hat{\beta}_{\mathcal{A}}^{SL} - \beta_{\mathcal{A}}^{*}\|_{2}^{2}} \\ & \geq &  1 - \frac{|\hat{\beta}_{\mathcal{A}}^{SL} - \beta_{\mathcal{A}}^{*}|_{1}^{2}}{\alpha |\mathcal{A}| \, \|\hat{\beta}_{\mathcal{A}}^{SL} - \beta_{\mathcal{A}}^{*}\|_{2}^{2}} -   \frac{2\mu_n \, \beta^{*'}_{\mathcal{A}}\widetilde{J}\beta_{\mathcal{A}}^{*} }{3 \alpha \lambda_n  |\mathcal{A}| } \frac{|\hat{\beta}_{\mathcal{A}}^{SL} - \beta_{\mathcal{A}}^{*}|_{1} }{ \|\hat{\beta}_{\mathcal{A}}^{SL} - \beta_{\mathcal{A}}^{*}\|_{2}^{2}} 
\\ 
& \geq  & ( 1-\frac{1}{\alpha} ) -   \frac{2\mu_n \, \beta^{*'}_{\mathcal{A}}\widetilde{J}\beta_{\mathcal{A}}^{*} }{3 \alpha \lambda_n  |\mathcal{A}| } \frac{|\hat{\beta}_{\mathcal{A}}^{SL} - \beta_{\mathcal{A}}^{*}|_{1} }{ \|\hat{\beta}_{\mathcal{A}}^{SL} - \beta_{\mathcal{A}}^{*}\|_{2}^{2}}.
\end{eqnarray*}
where we used the fact that \eqref{eq:pitsch2} implies $|\hat{\beta}_{\mathcal{A}^c}^{SL} - \beta_{\mathcal{A}^c}^{*}|_1   \leq 2 |\hat{\beta}_{\mathcal{A}}^{SL} - \beta_{\mathcal{A}}^{*}|_1  + 2 \frac{\mu_n}{\lambda_n} \beta^{*'}_{\mathcal{A}}\widetilde{J}\beta_{\mathcal{A}}^{*}$ in the third line. The last inequalities can be summed-up by \begin{equation}\label{eq:bizts}
(\hat{\beta}^{SL} - \beta^{*})'K_n(\hat{\beta}^{SL} - \beta^{*}) \geq (1-\frac{1}{\alpha}) \|\hat{\beta}_{\mathcal{A}}^{SL} - \beta_{\mathcal{A}}^{*}\|_{2}^{2} - \frac{2\mu_n \, \beta^{*'}_{\mathcal{A}}\widetilde{J}\beta_{\mathcal{A}}^{*} }{3 \alpha \lambda_n  |\mathcal{A}| } |\hat{\beta}_{\mathcal{A}}^{SL} - \beta_{\mathcal{A}}^{*}|_{1}.
\end{equation}
Let us consider \eqref{eq:toiu} and \eqref{eq:bizts}. An optimization work over $\|\hat{\beta}_{\mathcal{A}}^{SL} - \beta_{\mathcal{A}}^{*}\|_{2}$ provides us the following bound:
\begin{eqnarray}\label{eq:bornSurNorm2q2}
\|\hat{\beta}_{\mathcal{A}}^{SL} - \beta_{\mathcal{A}}^{*}\|_{2} \leq & (\frac{\alpha}{\alpha-1}) \left[  (\frac{3\lambda_n}{2} + 2 \mu_n \| \widetilde{J}\beta^{*} \|_{\infty} ) \sqrt{|\mathcal{A}|} + \frac{2\mu_n}{3 \alpha \lambda_n \sqrt{|\mathcal{A}| }} \beta^{*'}_{\mathcal{A}}\widetilde{J}\beta_{\mathcal{A}}^{*} \right] \nonumber \\ & + \sqrt{\frac{\alpha}{ \alpha-1} \left( \frac{3\lambda_n}{2} + 2 \mu_n \| \widetilde{J}\beta^{*} \|_{\infty}  \right) \frac{\mu_n}{\lambda_n}  \beta^{*'}_{\mathcal{A}}\widetilde{J}\beta_{\mathcal{A}}^{*}  }  .
\end{eqnarray}
Thanks to Assumption~(A1), $\beta^{*'}_{\mathcal{A}}\widetilde{J}\beta_{\mathcal{A}}^{*}\leq L_1\,\log{(p)} |\mathcal{A}|$ and $\| \widetilde{J}\beta^{*} \|_{\infty} \leq L_2  \,\log{(p)}$. Moreover the tuning parameters $\lambda_n$ and $\mu_n$ are chosen in the form $\lambda_n = \kappa_1 \sigma \sqrt{\log{(p)}/n}$ and $\mu_n= \kappa_3\sigma / n$. Then we conclude from \eqref{eq:tisch} and \eqref{eq:bornSurNorm2q2}
\begin{eqnarray*}
\| \hat{\beta}^{SL} - \beta^{*} \|_{\infty} \leq \frac{1}{1 + \frac{\kappa_3\sigma}{n}} & \left( \frac{3}{4} + \frac{1}{\alpha-1} + \frac{4 L_1 \kappa_3}{9 \alpha^2 \kappa_1^2}    + \frac{2 L_1 \kappa_3}{3 \alpha \kappa_1^2} + \sqrt{ \frac{2 L_1 \kappa_3}{3 \alpha(\alpha -1) \kappa_1^2} +\frac{8 L_1 \,L_2 \kappa_3^2}{9\alpha(\alpha-1) \kappa_1^4} \lambda_n} \right. \\  & \left. + (\frac{4 L_2 \kappa_3}{3 \kappa_1^2}+ \frac{L_2 \kappa_3}{\kappa_1^2}) \lambda_n \right) \lambda_n.
\end{eqnarray*}
This ends the proof.
\end{proof}
\vspace{0.5cm}
\begin{proof} [\it{Proof of Theorem~\ref{th:dll_exact}}]The proof of this theorem is essentially an adaptation of the one concerning the Lasso in \cite{Zou-df-Lasso}. We do not give the whole proof but only mention the
important steps and let the reader refer to \cite{Zou-df-Lasso} for more
details. The main points in the proof are Stein's lemma and these few facts:
\begin{itemize}
\item For every couple $(\lambda,\mu)$, the S-Lasso estimator is a continuous function of $Y$.
\item For every couple $(\lambda,\mu)=\zeta$, the active set
$\mathcal{A}_{\zeta}$ and the sign vector of $\hat{\beta}_{\zeta}^{SL}$
which we denote by ${\mathop{\rm Sgn}}_{\zeta}$ are piecewise constant with
respect to $Y$, out of a set with Lebesgue measure equal to $0$.
\end{itemize}
The detailed proof uses these points and the explicit form of the estimator
$\hat{\beta}^{SL}$ given by \eqref{eq:Explicit_beta}. This proof is the same as
the one in \cite{Zou-df-Lasso} so that we omit it here.
\end{proof}

\bibliographystyle{plain}
 \bibliography{S-LASSO}

\end{document}